\documentclass[12pt]{article}
\usepackage{graphicx}
\usepackage{hyperref}
\usepackage{mathrsfs}
\usepackage{amsfonts,amsmath,amsthm, amssymb}
\usepackage{latexsym, euscript, epic, eepic, color}
\usepackage{indentfirst}
\usepackage{color}
\usepackage{enumerate}
\usepackage{cite}
\makeatletter
\@addtoreset{equation}{section}
\makeatother
\newtheorem{theorem}{Theorem}[section]
\newtheorem{lemma}[theorem]{Lemma}

\newtheorem{corollary}[theorem]{Corollary}

\theoremstyle{definition}
\newtheorem{definition}[theorem]{Definition}

\newtheorem{remark}[theorem]{Remark}
\makeatletter
\newcommand{\rmnum}[1]{\romannumeral #1}
\newcommand{\Rmnum}[1]{\expandafter\@slowromancap\romannumeral #1@}
\makeatother

\begin{document}
\title{Ram\'{\i}rez's problems and fibers on well approximable set of systems of affine forms}

\author{Bing Li\textsuperscript{a} and Bo Wang\textsuperscript{b,c,}\thanks{Corresponding author}\\
\small \it \textsuperscript{\rm a }School of Mathematics, South China University of Technology,\\
\small \it Guangzhou 510641, China\\
\small \it \textsuperscript{\rm b }School of Mathematics, Sun Yat-sen University,\\
\small \it Guangzhou 510275, China\\
\small \it \textsuperscript{\rm c }School of Mathematics, Jiaying University,\\
\small \it Meizhou 514015, China
}
\date{}
\maketitle
\begin{center}
\begin{minipage}{120mm}
{\small {\bf Abstract.}
We show that badly approximable matrices are exactly those that, for any inhomogeneous parameter, can not be inhomogeneous approximated at every monotone divergent rate, which generalizes Ram\'{\i}rez's result (2018). We also establish some metrical results of the fibers on well approximable set of systems of affine forms, which gives answer to two of Ram\'{\i}rez's problems (2018). Furthermore, we prove that badly approximable systems are exactly those that, can not be approximated at each monotone convergent rate $\psi$. Moreover, we study the topological structure of the set of approximation functions.
}
\end{minipage}
\end{center}

\vskip0.5cm {\small{\bf Key words and phrases} \ Kurzweil's theorem, inhomogeneous Diophantine approximation, well approximable set, badly approximable systems, Baire category theorem}
\footnotetext{\small \it E-mails addresses: \rm scbingli@scut.edu.cn (B. Li), 
math\_bocomeon@163.com (B. Wang).}

\section{Introduction}
\subsection{Notations}
Firstly, we fix our notations in this paper. Let $m,n$ be two positive integers. Denote by $[0,1)^{m\times n}$ the set of all $m\times n$ matrices with entries in $[0,1)$. Given function $\psi:\mathbb{N}\to\mathbb{R}_{\geq0}$, where $\mathbb{R}_{\geq0}=[0,+\infty)$. Define
\begin{equation*}
W_{m,n}(\psi):=\left\{(A,\boldsymbol{\gamma})\in[0,1)^{m\times n}\times[0,1)^{m}:\langle A\boldsymbol{q}-\boldsymbol{\gamma} \rangle<\psi(\|\boldsymbol{q}\|)\ {\rm for\ i.m.}\ \boldsymbol{q}\in\mathbb{Z}^{n}\right\},
\end{equation*}
where $\|\cdot\|$ denotes supremum norm and $\langle\cdot\rangle$ denotes supremum norm distance to $\mathbb{Z}^{m}$, that is,
\begin{equation*}
\|\boldsymbol{x}\|=\max\left\{|x_{1}|,\cdots,|x_{n}|\right\}\ {\rm for\ any}\ \boldsymbol{x}=(x_{1},\cdots,x_{n})\in \mathbb{R}^{n}
\end{equation*}
and
\begin{equation*}
\langle\boldsymbol{y}\rangle=\min\left\{\|\boldsymbol{y}-\boldsymbol{p}\|:\boldsymbol{p}\in\mathbb{Z}^{m}\right\}\ {\rm for\ any}\ \boldsymbol{y}\in \mathbb{R}^{m}.
\end{equation*}
Here and throughout, ``i.m.'' stands for ``infinitely many''. The set $W_{m,n}(\psi)$ is usually called the collection of inhomogeneously $\psi$-approximable pairs in $[0,1)^{m\times n}\times[0,1)^{m}$. The corresponding fibers of $W_{m,n}(\psi)$ denoted by $W_{m,n}^{\boldsymbol{\gamma}}(\psi)$ and $W_{m,n,A}(\psi)$, that is,
\begin{equation*}
W_{m,n}^{\boldsymbol{\gamma}}(\psi):=\{A\in[0,1)^{m\times n}:(A,\boldsymbol{\gamma})\in W_{m,n}(\psi)\}
\end{equation*}
and
\begin{equation*}
W_{m,n,A}(\psi):=\{\boldsymbol{\gamma}\in[0,1)^{m}:(A,\boldsymbol{\gamma})\in W_{m,n}(\psi)\},
\end{equation*}
are called the well approximable set and the covering set respectively. Throughout this paper, for given $l\in\mathbb{N}$, we use $\mu_{l}$ for $l$ dimensional Lebesgue measure. The Lebesgue measure for $W_{m,n}^{\boldsymbol{\gamma}}(\psi)$ is summarized by the following statement, which called inhomogeneous Khintchine-Groshev theorem \cite[Theorem 12/15]{VGS}.
\begin{theorem}\label{01}{\rm(Inhomogeneous Khintchine-Groshev theorem)}
For any $\psi:\mathbb{N}\to\mathbb{R}_{\geq0}$ and $\boldsymbol{\gamma}\in[0,1)^{m}$, we have
\begin{equation*}
\mu_{mn}(W_{m,n}^{\boldsymbol{\gamma}}(\psi))=\begin{cases} 0, &\text{if}\ \sum_{q=1}^{\infty}q^{n-1}\psi(q)^{m}<\infty, \\
1, &\text{if}\ \sum_{q=1}^{\infty}q^{n-1}\psi(q)^{m}=\infty\ {\text and\ \psi\ is\ decreasing}. \end{cases}
\end{equation*}
\end{theorem}
For the Hausdorff measure and Hausdorff dimension of $W_{m,n}^{\boldsymbol{\gamma}}(\psi)$, see \cite{BBDV,JL}. The metric results of $W_{m,n,A}(\psi)$ is more complicated than those of $W_{m,n}^{\boldsymbol{\gamma}}(\psi)$, we refer the readers to \cite{YB,AJ,KRW,SJ}.

Throughout this paper, let $\mathcal{D}_{m,n}$ and $\mathcal{C}_{m,n}$ be the sets of all decreasing function $\psi:\mathbb{N}\to\mathbb{R}_{\geq0}$ satisfies $\sum\limits_{q=1}^{\infty}q^{n-1}\psi(q)^{m}$ diverges and converges, respectively. For the sake of simplicity, we will denote $\mathcal{D}_{m,n}$ and $\mathcal{C}_{m,n}$ by $\mathcal{D}$ and $\mathcal{C}$, respectively. Denote by
\begin{equation*}
\Omega(m,n):=\bigcap_{\psi\in\mathcal{D}}W_{m,n}(\psi).
\end{equation*}
Besides, we use $\Omega^{\boldsymbol{\gamma}}(m,n)$ and $\Omega_{A}(m,n)$ to denote the fibers of $\Omega(m,n)$. That is,
\begin{equation*}
\Omega^{\boldsymbol{\gamma}}(m,n)=\{A\in[0,1)^{m\times n}:(A,\boldsymbol{\gamma})\in\Omega(m,n)\}
\end{equation*}
and
\begin{equation*}
\Omega_{A}(m,n)=\{\boldsymbol{\gamma}\in[0,1)^{m}:(A,\boldsymbol{\gamma})\in\Omega(m,n)\}.
\end{equation*}
A dual set to $\Omega(m,n)$ is
\begin{equation*}
\Lambda(m,n)=\bigcup_{\psi\in\mathcal{C}}W_{m,n}(\psi).
\end{equation*}
The corresponding fibers of $\Lambda(m,n)$ denoted by $\Lambda^{\boldsymbol{\gamma}}(m,n)$ and $\Lambda_{A}(m,n)$, that is,
\begin{equation*}
\Lambda^{\boldsymbol{\gamma}}(m,n):=\{A\in[0,1)^{m\times n}: (A,\boldsymbol{\gamma})\in\Lambda_{m,n}\}= \bigcup_{\psi\in\mathcal{C}}W_{m,n}^{\boldsymbol{\gamma}}(\psi)
\end{equation*}
and
\begin{equation*}
\Lambda_{A}(m,n):=\{\boldsymbol{\gamma}\in[0,1)^{m}: (A,\boldsymbol{\gamma})\in\Lambda_{m,n}\}=\bigcup_{\psi\in\mathcal{C}}W_{m,n,A}(\psi).
\end{equation*}
Let us denote
\begin{equation*}
\boldsymbol{\rm Bad}(m,n):=\left\{(A,\boldsymbol{\gamma})\in[0,1)^{m\times n}\times[0,1)^{m}: \liminf\limits_{\boldsymbol{q}\in\mathbb{Z}^{n},\|\boldsymbol{q}\|\to\infty}\|\boldsymbol{q}\|^{n}\langle A\boldsymbol{q}-\boldsymbol{\gamma}\rangle^{m}>0\right\},
\end{equation*}
which is the set of badly approximable systems of $m$ affine forms in $n$ variables. We also denote the fibers of $\boldsymbol{\rm Bad}(m,n)$ by $\boldsymbol{\rm Bad}^{\boldsymbol{\gamma}}(m,n)$ and $\boldsymbol{\rm Bad}_{A}(m,n)$, that is,
\begin{equation*}
\boldsymbol{\rm Bad}^{\boldsymbol{\gamma}}(m,n):=\{A\in[0,1)^{m\times n}: (A,\boldsymbol{\gamma})\in\boldsymbol{\rm Bad}(m,n)\}
\end{equation*}
and
\begin{equation*}
\boldsymbol{\rm Bad}_{A}(m,n):=\{\boldsymbol{\gamma}\in[0,1)^{m}: (A,\boldsymbol{\gamma})\in\boldsymbol{\rm Bad}(m,n)\}.
\end{equation*}
For the metrical results of $\boldsymbol{\rm Bad}^{\boldsymbol{\gamma}}(m,n)$ and $\boldsymbol{\rm Bad}_{A}(m,n)$, see \cite{YHKV,MJ,WMS}.

\subsection{Ram\'{\i}rez's problems: concerning Kurzweil type theorem}
For $\psi:\mathbb{N}\to\mathbb{R}_{\geq0}$, denote $$V_{m,n}(\psi):=\{A\in[0,1)^{m\times n}: \mu_{m}(W_{m,n,A}(\psi))=1\}.$$
Motivated by Steinhaus's question: $$ {\rm whether\ all\ irrational\ numbers\ from\ the\ interval}\ [0,1)\ {\rm belong\ to\ the\ set}\ \bigcap\limits_{\psi\in\mathcal{D}}V_{1,1}(\psi)\ ?$$  In 1955, Kurzweil \cite{JK} established the following theorem.
\begin{theorem}{\rm(\cite[Theorem 5]{JK} )}
\begin{equation*}
\boldsymbol{\rm Bad}^{\boldsymbol{0}}(m,n)=\bigcap\limits_{\psi\in\mathcal{D}}V_{m,n}(\psi).
\end{equation*}
\end{theorem}
Based on the above result, Velani asked the following restricted question.\\
\textbf{Velani's question.} Let $M\subset[0,1)^{m}$ be some subset (say, an affine subspace, or any manifold, or a fractal) supporting a probability measure $\nu_{M}$. Do we still have $$\boldsymbol{\rm Bad}^{\boldsymbol{0}}(m,n)=\bigcap\limits_{\psi\in\mathcal{D}}V^{M}_{m,n}(\psi),$$
where $$V^{M}_{m,n}(\psi):=\{A\in[0,1)^{m\times n}: \nu_{M}(W_{m,n,A}(\psi)\cap M)=1\}.$$

To the best of our knowledge, the above \textbf{Velani's question} was first studied by Ram\'{\i}rez \cite{FAR1} who gave a answer to \textbf{Velani's question} when $n=1$ and $M=\{\boldsymbol{\gamma}\}$ is a single point with the Dirac measure $\nu=\delta_{\boldsymbol{\gamma}}$. Note that when $M=\{\boldsymbol{\gamma}\}$ is a single point with the Dirac measure $\nu=\delta_{\boldsymbol{\gamma}}$, since $\nu(W_{m,n,A}(\psi)\cap\{\boldsymbol{\gamma}\})=1$ is equivalent to $A\in W_{m,n}^{\boldsymbol{\gamma}}(\psi)$, we have $$V^{\{\boldsymbol{\gamma}\}}_{m,n}(\psi)=W_{m,n}^{\boldsymbol{\gamma}}(\psi).$$ It means that $$\bigcap\limits_{\psi\in\mathcal{D}}V^{\{\boldsymbol{\gamma}\}}_{m,n}(\psi)=\bigcap\limits_{\psi\in\mathcal{D}}W_{m,n}^{\boldsymbol{\gamma}}(\psi)=\Omega^{\boldsymbol{\gamma}}(m,n).$$ Ram\'{\i}rez \cite[Theorem 2.1]{FAR1} showed that
\begin{equation*}
\Omega^{\boldsymbol{\gamma}}(m,1)\cap\boldsymbol{\rm Bad}^{\boldsymbol{0}}(m,1)=\emptyset\ {\rm for\ each}\ \boldsymbol{\gamma}\in[0,1)^{m}.
\end{equation*}
More exactly, Ram\'{\i}rez \cite[Theorem 2.1]{FAR1} proved that
\begin{equation*}
\bigcup_{\boldsymbol{\gamma}\in[0,1)^{m}}\Omega^{\boldsymbol{\gamma}}(m,1)=[0,1)^{m}\setminus\boldsymbol{\rm Bad}^{\boldsymbol{0}}(m,1).
\end{equation*}
What is more, Ram\'{\i}rez \cite[Theorem 2.2]{FAR1} demonstrated that $\Omega(m,1)$ is Lebesgue measurable and $\Omega_{A}(m,1)$ has Lebesgue measure 0 for each $A\in[0,1)^{m}$. By Ram\'{\i}rez's results and Fubini's theorem \cite[Theorem 18.3]{PB}, we know that
\begin{equation*}
\Omega_{A}(m,1)\neq\emptyset\ \Leftrightarrow\ A\in[0,1)^{m}\setminus\boldsymbol{\rm Bad}^{\boldsymbol{0}}(m,1)
\end{equation*}
and
\begin{equation*}
\Omega^{\boldsymbol{\gamma}}(m,1)\ {\rm has\ Lebesgue\ measure\ 0\ for\ almost\ all}\ \boldsymbol{\gamma}\in[0,1)^{m}.
\end{equation*}
Ram\'{\i}rez also proved that $\Omega^{\boldsymbol{\gamma}}(m,1)$ always have measure 0 or 1 \cite[Lemma 4.5]{FAR1}. Furthermore, Ram\'{\i}rez gave some open problems in \cite{FAR1}.\\
\textbf{Problem 1.} Whether does there exist $\boldsymbol{\gamma}\in[0,1)^{m}$ for which $\Omega^{\boldsymbol{\gamma}}(m,1)$ has measure 1 ? \\
\textbf{Problem 2.} Is $\Omega^{\boldsymbol{\gamma}}(m,1)$ non-empty for any $\boldsymbol{\gamma}\in[0,1)^{m}$ ?

In this paper, we give a negative answer to \textbf{Problem 1} and give an affirmative answer to \textbf{Problem 2}. Our main result are as follows.
\begin{theorem}\label{27}
\begin{enumerate}[(i)]
\item We have $$\bigcup_{\boldsymbol{\gamma}\in[0,1)^{m}}\Omega^{\boldsymbol{\gamma}}(m,n)= [0,1)^{m\times n}\setminus\boldsymbol{\rm Bad}^{\boldsymbol{0}}(m,n).$$
\item For any $\boldsymbol{\gamma}\in[0,1)^{m}$, we have $$\mu_{mn}(\Omega^{\boldsymbol{\gamma}}(m,n))=0.$$
\item For every $A\in[0,1)^{m\times n}$, we have $$\mu_{m}(\Omega_{A}(m,n))=0.$$
\end{enumerate}
\end{theorem}
\begin{remark}
Theorem $\ref{27}$ (\rmnum{2}) with $n=1$ gives a negative answer to \textbf{Problem 1}. Moreover, the answer to \textbf{Problem 1} is ``No'' for $n\geq1$.
\end{remark}
By Theorem $\ref{27}$, we know that $\mu_{mn}(\Omega^{\boldsymbol{\gamma}}(m,n))=0$ for any $\boldsymbol{\gamma}\in[0,1)^{m}$ and $\mu_{m}(\Omega_{A}(m,n))=0$ for every $A\in[0,1)^{m\times n}$. Naturally, we want to know the Hausdorff dimension of the sets $\Omega^{\boldsymbol{\gamma}}(m,n)$ and $\Omega_{A}(m,n)$. Theorem $\ref{114}$ gives the Hausdorff dimension of $\Omega^{\boldsymbol{0}}(m,n)$ when $mn>1$ and a lower bound of the Hausdorff dimension of $\Omega^{\boldsymbol{\gamma}}(m,n)$ when $m>n$. Theorem $\ref{120}$ gives a upper bound of the Hausdorff dimension of $\Omega_{\alpha}(1,1)$. Furthermore, we can obtain the Hausdorff dimension of $\Omega_{\alpha}(1,1)$ when the denominator of the convergent of the continued fraction of $\alpha$ increases super-exponentially.
\begin{theorem}\label{114}
(\rmnum{1}) If $m,n\in\mathbb{N}$ and $mn>1$, then we have $$\dim_{\rm H}(\Omega^{\boldsymbol{0}}(m,n))=mn\left(1-\frac{1}{m+n}\right).$$
(\rmnum{2}) If $m,n\in\mathbb{N}$ and $m>n$, then for all $\boldsymbol{\gamma}\in[0,1)^{m}$, we have $$\dim_{\rm H}(\Omega^{\boldsymbol{\gamma}}(m,n))\geq m(n-1)+m\left(\frac{m-n}{m+n}\right)^{2}.$$
\end{theorem}
In view of Theorem $\ref{114}$ (\rmnum{2}), we immediately obtain the following corollary. The condition of Corollary $\ref{115}$ is weaker than the condition of Theorem $\ref{114}$ (\rmnum{2}).
\begin{corollary}\label{115}
For all $m,n\in\mathbb{N}$ with $mn>1$, we have $$\Omega^{\boldsymbol{\gamma}}(m,n)\neq\emptyset$$ for each $\boldsymbol{\gamma}\in[0,1)^{m}$.
\end{corollary}

\begin{remark}
Corollary $\ref{115}$ gives an affirmative answer to \textbf{Problem 2} when $mn>1$.
The remaining unknown case is $m=n=1$.
\end{remark}

\begin{remark}
For $m=n=1$, it follows from \cite[Proposition A.2]{FAR1} that
\begin{equation*}
\Omega^{\gamma}(1,1)=\left\{\frac{p+\gamma}{q}:q\in\mathbb{N},0\leq p\leq q-1\right\},\ \forall\ \gamma\in[0,1)\cap\mathbb{Q}.
\end{equation*}
When $m>1$, in view of Theorem $\ref{114}$ (\rmnum{2}),
\begin{equation*}
\dim_{\rm H}\Omega^{\boldsymbol{\gamma}}(m,1)\geq m\left(\frac{m-1}{m+1}\right)^{2}\geq\frac{2}{9},\ \forall\ \boldsymbol{\gamma}\in[0,1)^{m}.
\end{equation*}
This implies that there is a significant difference between $m=1$ and $m\geq2$ in terms of Hausdorff dimension since the set $\Omega^{\gamma}(1,1)$ is countable when $\gamma\in[0,1)\cap\mathbb{Q}$.
\end{remark}
Now we estimate the Hausdorff dimension of $\Omega_{A}(m,n)$ with $m=n=1$, $A=\alpha\in[0,1)\setminus\mathbb{Q}$. For $\alpha\in[0,1)\setminus\mathbb{Q}$, denote $$w(\alpha):=\sup\{s>0: \langle q\alpha\rangle<q^{-s}\ {\rm for\ i.m.}\ q\in\mathbb{N}\},$$ which is the irrationality exponent of $\alpha$.
\begin{theorem}\label{120}
(\rmnum{1}) For any $\alpha\in[0,1)\setminus\mathbb{Q}$, we have $$\dim_{\rm H}(\Omega_{\alpha}(1,1))\leq\frac{2}{w(\alpha)+1}.$$
(\rmnum{2}) Let $\alpha\in[0,1)\setminus\mathbb{Q}$ with $$\liminf_{k\to\infty}\frac{\log q_{k+1}}{\log q_{k}}>1,$$ where $q_{k}=q_{k}(\alpha)$ is the denominator of the $k$-th convergent of the continued fraction of $\alpha$. Then $$\dim_{\rm H}(\Omega_{\alpha}(1,1))=\frac{1}{w(\alpha)+1}.$$
\end{theorem}

\begin{remark}
Since
\begin{equation*}
\{\gamma\in[0,1):\Omega^{\gamma}(1,1)\neq\emptyset\}=\bigcup_{\alpha\in[0,1)}\Omega_{\alpha}(1,1),
\end{equation*}
it follows from Theorem $\ref{120}$ (\rmnum{2}) that
\begin{equation*}
\dim_{\rm H}(\{\gamma\in[0,1):\Omega^{\gamma}(1,1)\neq\emptyset\})\geq\frac{1}{2}.
\end{equation*}
\end{remark}

\begin{remark}
For any $\alpha\in[0,1)\cap\mathbb{Q}$, it follows from \cite[Proposition A.2]{FAR1} that
\begin{equation*}
\Omega_{\alpha}(1,1)=\{q\alpha+\lceil-q\alpha\rceil:q\in\mathbb{N}\},
\end{equation*}
where $\lceil x\rceil$ is the smallest integer not less than $x$.
\end{remark}

\begin{remark}
For any $\alpha\in[0,1)\setminus\mathbb{Q}$ with $w(\alpha)=+\infty$ (i.e. $\alpha$ is a Liouville number), it follows from Theorem $\ref{120}$ (\rmnum{1}) that $$\dim_{\rm H}(\Omega_{\alpha}(1,1))=0.$$
\end{remark}


\begin{remark}
Following the method in Ram\'{\i}rez's paper \cite[Lemmas 4.1 and 4.2]{FAR1}, we also can prove that $\Omega(m,n)$ is Lebesgue measurable. Since $\Omega(m,n)$ is not the main research object in this article, we omit the proof here. We refer the readers to \cite{FAR1} for more details. By Theorem $\ref{27}$ (\rmnum{2}) (or (\rmnum{3})) and Fubini's theorem, $\Omega(m,n)$ is a Lebesgue null set. Furthermore, denote by $\pi$ the projection $[0,1)^{m\times n}\times[0,1)^{m}\to[0,1)^{m\times n}$ to the first copy of $[0,1)^{m\times n}$. That is, $\pi(A,\boldsymbol{\gamma})=A$. Then $\pi(\Omega(m,n))=\bigcup\limits_{\boldsymbol{\gamma}\in[0,1)^{m}}\Omega^{\boldsymbol{\gamma}}(m,n)$. By Theorem $\ref{27}$ (\rmnum{1}), we have $$\dim_{\rm H}(\Omega(m,n))\geq mn.$$ What is more, in view of Theorem $\ref{27}$ (\rmnum{1}), Theorem $\ref{114}$ (\rmnum{2}) and \cite[Corollary 7.12]{KJF}, we obtain that for every $m,n\in\mathbb{N}$, $$\dim_{\rm H}(\Omega(m,n))\geq mn+m\left(\max\left(0,\frac{m-n}{m+n}\right)\right)^{2}.$$
\end{remark}
\subsection{The dual problem to Kurzweil type theorem}
The set $\Lambda(m,n)$ is a dual set to $\Omega(m,n)$, a natural question to Theorem $\ref{27}$ is how large the corresponding fiber sets $\Lambda^{\boldsymbol{\gamma}}(m,n)$ and $\Lambda_{A}(m,n)$ are. The following gives a complete characterization of $\Lambda(m,n)$.
\begin{theorem}\label{12}
We have $$\Lambda(m,n)=([0,1)^{m\times n}\times[0,1)^{m})\setminus\boldsymbol{\rm Bad}(m,n).$$
\end{theorem}
By Theorem $\ref{12}$, we immediately get the following corollary, expressing Theorem $\ref{12}$ by fibers.
\begin{corollary}\label{13}
\begin{enumerate}[(i)]
\item For any $A\in[0,1)^{m\times n}$, we have
\begin{equation*}
 \Lambda_{A}(m,n)=[0,1)^{m}\setminus\boldsymbol{\rm Bad}_{A}(m,n) .
\end{equation*}
\item For any $\boldsymbol{\gamma}\in[0,1)^{m}$, we have
\begin{equation*}
 \Lambda^{\boldsymbol{\gamma}}(m,n)=[0,1)^{m\times n}\setminus \boldsymbol{\rm Bad}^{\boldsymbol{\gamma}}(m,n).
\end{equation*}
\end{enumerate}
\end{corollary}
\begin{remark}
It follows from Theorem $\ref{01}$ (Inhomogeneous Khintchine-Groshev theorem) that $\boldsymbol{\rm Bad}^{\boldsymbol{\gamma}}(m,n)$ is a Lebesgue null set for any $\boldsymbol{\gamma}\in[0,1)^{m}$. Since $\boldsymbol{\rm Bad}(m,n)$ is measurable, by Fubini's theorem, we know that $\boldsymbol{\rm Bad}(m,n)$ has measure zero and $\boldsymbol{\rm Bad}_{A}(m,n)$ has measure zero for almost all $A\in[0,1)^{m\times n}$. Hence,
\begin{equation*}
\mu_{m}(\Lambda_{A}(m,n))=1\ {\rm for\ almost\ all}\ A\in[0,1)^{m\times n}
\end{equation*}
and
\begin{equation*}
\mu_{mn}(\Lambda^{\boldsymbol{\gamma}}(m,n))=1\ {\rm for\ any}\ \boldsymbol{\gamma}\in[0,1)^{m}.
\end{equation*}
What is more, Bugeaud, Harrap, Kristensen and Velani \cite[Theorem 1]{YHKV} showed that $\dim_{\rm H}(\boldsymbol{\rm Bad}_{A}(m,n))=m$ for any $A\in[0,1)^{m\times n}$. Einsiedler and Tseng \cite[Theorem 1.1]{MJ} proved that $\dim_{\rm H}(\boldsymbol{\rm Bad}^{\boldsymbol{\gamma}}(m,n))=mn$ for every $\boldsymbol{\gamma}\in[0,1)^{m}$. It follows that
\begin{equation*}
\dim_{\rm H}([0,1)^{m}\setminus\Lambda_{A}(m,n))=m\ {\rm for\ any}\ A\in[0,1)^{m\times n}
\end{equation*}
 and
\begin{equation*}
\dim_{\rm H}([0,1)^{m\times n}\setminus\Lambda^{\boldsymbol{\gamma}}(m,n))=mn\ {\rm for\ every}\ \boldsymbol{\gamma}\in[0,1)^{m}.
\end{equation*}
\end{remark}
\subsection{Topological property for the set of approximation functions}
Recall that $\mathcal{C}$ is the set of all decreasing function $\psi:\mathbb{N}\to\mathbb{R}_{\geq0}$ such that the infinite series $\sum\limits_{q=1}^{\infty}q^{n-1}\psi(q)^{m}$ converges. The set $W_{m,n}(\psi)$ describes the set of well approximable pairs $(A,\boldsymbol{\gamma})$ for the given function $\psi$. A relative problem is how about the set of functions $\psi$ for given $(A,\boldsymbol{\gamma})\in[0,1)^{m\times n}\times[0,1)^{m}$.
Define
\begin{equation}\label{60}
\mathcal{C}(A,\boldsymbol{\gamma}):=\{\psi\in\mathcal{C}:\langle A\boldsymbol{q}-\boldsymbol{\gamma} \rangle<\psi(\|\boldsymbol{q}\|)\ {\rm for\ i.m.}\ \boldsymbol{q}\in\mathbb{Z}^{n}\}.
\end{equation}
To measure the sets in $\mathcal{C}$, we first define a reasonable metric. Let $d:\mathcal{C}\times\mathcal{C}\to[0,+\infty)$,
\begin{equation*}
d(\psi_{1},\psi_{2}):=\sum\limits_{q=1}^{\infty}q^{n-1}|\psi_{1}(q)^{m}-\psi_{2}(q)^{m}|,\ \forall\ \psi_{1},\psi_{2}\in\mathcal{C}.
\end{equation*}
Clearly, $d$ is a metric in $\mathcal{C}$ and the metric space $(\mathcal{C},d)$ is complete. A natural question is how large the set $\mathcal{C}(A,\boldsymbol{\gamma})$ is in the topological sense. From Theorem $\ref{12}$, we immediately obtain that the following corollary.
\begin{corollary}\label{10}
For any $(A,\boldsymbol{\gamma})\in[0,1)^{m\times n}\times[0,1)^{n}$, we have
\begin{equation*}
\mathcal{C}(A,\boldsymbol{\gamma})\neq\emptyset \quad \Leftrightarrow \quad (A,\boldsymbol{\gamma})\notin\boldsymbol{\rm Bad}(m,n).
\end{equation*}
\end{corollary}
\begin{remark}
Corollary $\ref{10}$ is just another statement of Theorem $\ref{12}$.
\end{remark}
The following theorem shows some topological properties of $\mathcal{C}(A,\boldsymbol{\gamma})$ when $\mathcal{C}(A,\boldsymbol{\gamma})$ is not an empty set.
\begin{theorem}\label{14}
For any $(A,\boldsymbol{\gamma})\in\left([0,1)^{m\times n}\times[0,1)^{m}\right)\setminus\boldsymbol{\rm Bad}(m,n)$, we have
\begin{enumerate}[(i)]
\item $\mathcal{C}(A,\boldsymbol{\gamma})$ is a $G_{\delta}$ set and dense in $\mathcal{C}$;
\item $\mathcal{C}(A,\boldsymbol{\gamma})$ is not a $F_{\sigma}$ set in $\mathcal{C}$.
\end{enumerate}
\end{theorem}
\begin{remark}
In view of Theorem $\ref{14}$ (\rmnum{1}) and Baire category theorem (see Section $\ref{600}$), we know that $\mathcal{C}(A,\boldsymbol{\gamma})$ is of second category in $\mathcal{C}$. Theorem $\ref{14}$ implies that in case $\mathcal{C}(A,\boldsymbol{\gamma})\neq\emptyset$, the set $\mathcal{C}(A,\boldsymbol{\gamma})$ is ``large'' in the sense of topology.
\end{remark}
From Theorem $\ref{14}$ and Baire category theorem, we obtain the following corollary.
\begin{corollary}\label{15}
For any $\{(A_{i},\boldsymbol{\gamma}_{i})\}_{i=1}^{\infty}\subset\left([0,1)^{m\times n}\times[0,1)^{m}\right)\setminus\boldsymbol{\rm Bad}(m,n)$, we have
\begin{equation*}
\bigcap_{i=1}^{\infty}\mathcal{C}(A_{i},\boldsymbol{\gamma}_{i})\ {is\ of\ second\ category\ and\ dense\ in}\ \mathcal{C}.
\end{equation*}
\end{corollary}
\begin{remark}
Corollary $\ref{15}$ implies that for any $\{(A_{i},\boldsymbol{\gamma}_{i})\}_{i=1}^{\infty}\subset[0,1)^{m\times n}\times[0,1)^{m}$ with $\liminf\limits_{\boldsymbol{q}\in\mathbb{Z}^{n},\|\boldsymbol{q}\|\to\infty}\|\boldsymbol{q} \|^{n}\langle A_{i}\boldsymbol{q}-\boldsymbol{\gamma_{i}}\rangle^{m}=0$, there exists $\psi\in \mathcal{C}$, such that for each $i\in\mathbb{N}$, $\langle A_{i}\boldsymbol{q}-\boldsymbol{\gamma}_{i}\rangle<\psi(\|\boldsymbol{q}\|)$ for infinitely many $\boldsymbol{q}\in\mathbb{Z}^{n}$.
\end{remark}
The remainder of this paper is organized as follows. In Section $\ref{200}$, we give some lemmas, which play an important role in the proof of Theorems $\ref{27}$, $\ref{114}$, $\ref{120}$ and $\ref{12}$. In Section $\ref{300}$, we give the proof of Theorem $\ref{27}$. The proof of Theorem $\ref{114}$, Theorem $\ref{120}$ and Corollary $\ref{115}$ are given in Section $\ref{400}$. In Section $\ref{500}$, we prove Theorem $\ref{12}$. Section $\ref{600}$ is dedicated to the proof of Theorem $\ref{14}$ and Corollary $\ref{15}$.

\section{Preliminary}\label{200}
In this section, we give some lemmas, which take an important part in the proof of Theorems $\ref{27}$, $\ref{114}$, $\ref{120}$ and $\ref{12}$.

For $(A,\boldsymbol{\gamma})\in[0,1)^{m\times n}\times[0,1)^{m}$ and $l\in\mathbb{N}$, let
\begin{equation*}
S_{l}(A,\boldsymbol{\gamma}):=\sum_{t=l}^{\infty}t^{n-1}\cdot\left(\min_{\boldsymbol{q}\in\mathbb{Z}^{n},l\leq\|\boldsymbol{q}\|\leq t}\langle A\boldsymbol{q}-\boldsymbol{\gamma}\rangle^{m}\right).
\end{equation*}
These sums will take an essential role here, especially through the use of Lemma $\ref{36}$. In the following, for each $l\in\mathbb{N}$, we will denote $\min\limits_{\boldsymbol{q}\in\mathbb{Z}^{n},l\leq\|\boldsymbol{q}\|\leq t}$ by $\min\limits_{l\leq\|\boldsymbol{q}\|\leq t}$ without confusion. The following lemma gives a characterization of $\Omega(m,n)$ by the convergence of the series $S_{l}(A,\boldsymbol{\gamma})$.
\begin{lemma}\label{36}
For any $(A,\boldsymbol{\gamma})\in[0,1)^{m\times n}\times[0,1)^{m}$, we have
\begin{equation*}
(A,\boldsymbol{\gamma})\in \Omega(m,n) \quad {\rm if\ and\ only\ if} \quad S_{l}(A,\boldsymbol{\gamma})<+\infty\ {\rm for\ all}\ l\in\mathbb{N}.
\end{equation*}
\end{lemma}
\begin{proof}[\rm \textbf{Proof}] Firstly we prove the ``only if'' part by contradiction. Let $(A,\boldsymbol{\gamma})\in\Omega(m,n)$. Suppose that there exists $l_{0}\in\mathbb{N}$ such that $S_{l_{0}}(A,\boldsymbol{\gamma})=+\infty$. Choose a decreasing function $\psi_{0}:\mathbb{N}\to\mathbb{R}_{\geq0}$ satisfying
\begin{equation*}
\psi_{0}(t)=\min_{l_{0}\leq\|\boldsymbol{q}\|\leq t}\langle A\boldsymbol{q}-\boldsymbol{\gamma}\rangle \quad {\rm for\ all}\ t\geq l_{0}.
\end{equation*}
Then $\sum\limits_{t=l_{0}}^{+\infty}t^{n-1}\psi_{0}(t)^{m}=S_{l_{0}}(A,\boldsymbol{\gamma})=+\infty$, which gives $\psi_{0}\in\mathcal{D}$. Note that for any $\|\boldsymbol{q}\|\geq l_{0}$, $\psi_{0}(\|\boldsymbol{q}\|)\leq\langle A\boldsymbol{q}-\boldsymbol{\gamma}\rangle$, we have $(A,\boldsymbol{\gamma})\notin W_{m,n}(\psi_{0})$. It follows that $(A,\boldsymbol{\gamma})\notin\Omega(m,n)$, which contradicts with $(A,\boldsymbol{\gamma})\in\Omega(m,n)$. Therefore $S_{l}(A,\boldsymbol{\gamma})<\infty$ for each $l\in\mathbb{N}$. \\ Secondly we prove the ``if'' part. Given any $\psi\in\mathcal{D}$. For any fixed $l\in\mathbb{N}$, since the series $S_{l}(A,\boldsymbol{\gamma})=\sum\limits_{t=l}^{+\infty}t^{n-1}\left(\min\limits_{l\leq\|\boldsymbol{q}\|\leq t}\langle A\boldsymbol{q}-\boldsymbol{\gamma}\rangle\right)^{m}$ converges and $\sum\limits_{t=l}^{+\infty}t^{n-1}\psi(t)^{m}$ diverges, we know that there exists $t_{l}\geq l$, such that
\begin{equation*}
t_{l}^{n-1}\left(\min_{l\leq\|\boldsymbol{q}\|\leq t_{l}}\langle A\boldsymbol{q}-\boldsymbol{\gamma}\rangle\right)^{m}<t_{l}^{n-1}\psi(t_{l})^{m},
\end{equation*}
that is,
\begin{equation*}
\min_{l\leq\|\boldsymbol{q}\|\leq t_{l}}\langle A\boldsymbol{q}-\boldsymbol{\gamma}\rangle<\psi(t_{l}).
\end{equation*}
Denote by $\boldsymbol{q}_{l}$ with $l\leq\|\boldsymbol{q}_{l}\|\leq t_{l}$ and
\begin{equation*}
\langle A\boldsymbol{q}_{l}-\boldsymbol{\gamma}\rangle=\min_{l\leq\|\boldsymbol{q}\|\leq t_{l}}\langle A\boldsymbol{q}-\boldsymbol{\gamma}\rangle.
\end{equation*}
So $\langle A\boldsymbol{q}_{l}-\boldsymbol{\gamma}\rangle<\psi(t_{l})$. Since $\psi$ is non-increasing, we have
\begin{equation*}
\langle A\boldsymbol{q}_{l}-\boldsymbol{\gamma}\rangle<\psi(t_{l})\leq\psi(\|\boldsymbol{q}_{l}\|).
\end{equation*}
That is, $\langle A\boldsymbol{q}_{l}-\boldsymbol{\gamma}\rangle<\psi(\|\boldsymbol{q}_{l}\|)$. The fact $\|\boldsymbol{q}_{l}\|\geq l$ and arbitrariness of $l\in\mathbb{N}$ can guarantee that there exists infinitely many $\boldsymbol{q}\in\mathbb{Z}^{n}$ such that $\langle A\boldsymbol{q}-\boldsymbol{\gamma}\rangle<\psi(\|\boldsymbol{q}\|)$. Thus $(A,\boldsymbol{\gamma})\in W_{m,n}(\psi)$, which implies that $(A,\boldsymbol{\gamma})\in\Omega(m,n)$ by the arbitrariness of $\psi\in\mathcal{D}$.
\textrm{}
\end{proof}

\begin{lemma}\label{43}
If $\langle A\boldsymbol{q}-\boldsymbol{\gamma}\rangle>0$ for any $\boldsymbol{q}\in\mathbb{Z}^{n}\setminus\{\boldsymbol{0}\}$, then the series $S_{l}(A,\boldsymbol{\gamma})$ either converges for all $l\in\mathbb{N}$, or diverges for all $l\in\mathbb{N}$.
\end{lemma}
\begin{proof}[\rm \textbf{Proof}]
Suppose that there exists $l_{0}\in\mathbb{N}$ such that $S_{l_{0}}(A,\boldsymbol{\gamma})<+\infty$, we show that $S_{l}(A,\boldsymbol{\gamma})<+\infty$ for all $l\in\mathbb{N}$. For any fixed $l\in\mathbb{N}$, we consider two cases. \\ Case 1: $l\leq l_{0}$. Then
\begin{equation*}
\begin{aligned}
S_{l}(A,\boldsymbol{\gamma})&=\sum_{t=l}^{\infty}t^{n-1}\min_{l\leq\|\boldsymbol{q}\|\leq t}\langle A\boldsymbol{q}-\boldsymbol{\gamma}\rangle^{m}\\
&=\sum_{t=l}^{l_{0}}t^{n-1}\min_{l\leq\|\boldsymbol{q}\|\leq t}\langle A\boldsymbol{q}-\boldsymbol{\gamma}\rangle^{m}+\sum_{t=l_{0}+1}^{\infty}t^{n-1}\min_{l\leq\|\boldsymbol{q}\|\leq t}\langle A\boldsymbol{q}-\boldsymbol{\gamma}\rangle^{m} \\
&\leq\sum_{t=l}^{l_{0}}t^{n-1}\min_{l\leq\|\boldsymbol{q}\|\leq t}\langle A\boldsymbol{q}-\boldsymbol{\gamma}\rangle^{m} +\sum_{t=l_{0}}^{\infty}t^{n-1}\min_{l_{0}\leq\|\boldsymbol{q}\|\leq t}\langle A\boldsymbol{q}-\boldsymbol{\gamma}\rangle^{m} \\
&=\sum_{t=l}^{l_{0}}t^{n-1}\min_{l\leq\|\boldsymbol{q}\|\leq t}\langle A\boldsymbol{q}-\boldsymbol{\gamma}\rangle^{m}+S_{l_{0}}(A,\boldsymbol{\gamma})<+\infty.
\end{aligned}
\end{equation*}
Case 2: $l>l_{0}$. Since $\langle A\boldsymbol{q}-\boldsymbol{\gamma}\rangle>0$ for any $\boldsymbol{q}\in\mathbb{Z}^{n}\setminus\{\boldsymbol{0}\}$, we have $$\min\limits_{l_{0}\leq\|\boldsymbol{q}\|\leq l-1}\langle A\boldsymbol{q}-\boldsymbol{\gamma}\rangle^{m}>0.$$ The convergence of the series
$S_{l_{0}}(A,\boldsymbol{\gamma})=\sum\limits_{t=l_{0}}^{\infty}t^{n-1}\min\limits_{l_{0}\leq\|\boldsymbol{q}\|\leq t}\langle A\boldsymbol{q}-\boldsymbol{\gamma}\rangle^{m}$ implies that $$\lim\limits_{t\to\infty}\min\limits_{l_{0}\leq\|\boldsymbol{q}\|\leq t}\langle A\boldsymbol{q}-\boldsymbol{\gamma}\rangle^{m}=0.$$ Thus, there exists $L\geq l+1$, such that for all $t\geq L$, we have
\begin{equation*}
\min\limits_{l_{0}\leq\|\boldsymbol{q}\|\leq t}\langle A\boldsymbol{q}-\boldsymbol{\gamma}\rangle^{m}<\min\limits_{l_{0}\leq\|\boldsymbol{q}\|\leq l-1}\langle A\boldsymbol{q}-\boldsymbol{\gamma}\rangle^{m}.
\end{equation*}
It follows that
\begin{equation*}
\min\limits_{l_{0}\leq\|\boldsymbol{q}\|\leq t}\langle A\boldsymbol{q}-\boldsymbol{\gamma}\rangle^{m}=\min\limits_{l\leq\|\boldsymbol{q}\|\leq t}\langle A\boldsymbol{q}-\boldsymbol{\gamma}\rangle^{m},\ \forall\ t\geq L.
\end{equation*}
Therefore
\begin{equation*}
\begin{aligned}
S_{l}(A,\boldsymbol{\gamma})&=\sum_{t=l}^{\infty}t^{n-1}\min_{l\leq\|\boldsymbol{q}\|\leq t}\langle A\boldsymbol{q}-\boldsymbol{\gamma}\rangle^{m}\\
&=\sum_{t=l}^{L-1}\min_{l\leq\|\boldsymbol{q}\|\leq t}\langle A\boldsymbol{q}-\boldsymbol{\gamma}\rangle^{m}+\sum_{t=L}^{\infty}t^{n-1}\min_{l\leq\|\boldsymbol{q}\|\leq t}\langle A\boldsymbol{q}-\boldsymbol{\gamma}\rangle^{m}\\
&=\sum_{t=l}^{L-1}\min_{l\leq\|\boldsymbol{q}\|\leq t}\langle A\boldsymbol{q}-\boldsymbol{\gamma}\rangle^{m}+\sum_{t=L}^{\infty}t^{n-1}\min_{l_{0}\leq\|\boldsymbol{q}\|\leq t}\langle A\boldsymbol{q}-\boldsymbol{\gamma}\rangle^{m}\\ &\leq\sum_{t=l}^{L-1}\min_{l\leq\|\boldsymbol{q}\|\leq t}\langle A\boldsymbol{q}-\boldsymbol{\gamma}\rangle^{m}+S_{l_{0}}(A,\boldsymbol{\gamma})<+\infty.
\end{aligned}
\end{equation*}
\textrm{}
\end{proof}
By Lemma $\ref{36}$ and Lemma $\ref{43}$, we immediately obtain the following corollary.
\begin{corollary}\label{41}
If $\langle A\boldsymbol{q}-\boldsymbol{\gamma}\rangle>0$ for any $\boldsymbol{q}\in\mathbb{Z}^{n}\setminus\{\boldsymbol{0}\}$, then
\begin{equation*}
(A,\boldsymbol{\gamma})\in \Omega(m,n) \quad {\rm if\ and\ only\ if} \quad S_{1}(A,\boldsymbol{\gamma})<+\infty.
\end{equation*}
\end{corollary}
The following lemma is a simple consequence of bad approximability.
\begin{lemma}\label{44}
Suppose $A\in\boldsymbol{\rm Bad}^{\boldsymbol{0}}(m,n)$. Then there exists a positive constant $c:=c(A)$ such that for any $\epsilon>0$,
\begin{equation*}
\min\{\|\boldsymbol{q}\|: \boldsymbol{q}\in\mathbb{Z}^{n}\setminus\{\boldsymbol{0}\}, \langle A\boldsymbol{q}\rangle<\epsilon\}\geq \left(\frac{c}{\epsilon^{m}}\right)^{\frac{1}{n}}.
\end{equation*}
\end{lemma}
\begin{proof}[\rm \textbf{Proof}]
Since $A\in\boldsymbol{\rm Bad}^{\boldsymbol{0}}(m,n)$, there exists $c:=c(A)>0$, such that
\begin{equation*}
\|\boldsymbol{q}\|^{n}\langle A\boldsymbol{q}\rangle^{m}\geq c\  {\rm for\ any}\ \boldsymbol{q}\in\mathbb{Z}^{n}\setminus\{\boldsymbol{0}\}.
\end{equation*}
Therefore for all $\boldsymbol{q}\in\mathbb{Z}^{n}\setminus\{\boldsymbol{0}\}$ with $\langle A\boldsymbol{q}\rangle<\epsilon$, we have
\begin{equation*}
\|\boldsymbol{q}\|\geq \left(\frac{c}{\epsilon^{m}}\right)^{\frac{1}{n}},
\end{equation*}
which implies that
\begin{equation*}
\min\{\|\boldsymbol{q}\|: \boldsymbol{q}\in\mathbb{Z}^{n}\setminus\{\boldsymbol{0}\}, \langle A\boldsymbol{q}\rangle<\epsilon\}\geq \left(\frac{c}{\epsilon^{m}}\right)^{\frac{1}{n}}.
\end{equation*}
\textrm{}
\end{proof}
The proof of Lemmas $\ref{36}$, $\ref{43}$ and $\ref{44}$ are similar to the proof of \cite[Lemmas 3.1, 3.2 and 3.3]{FAR1}, which consider the case $n=1$. For completeness, we still give the above proof. The following elementary lemma, which is a generalization of the Olivier's theorem \cite{MLO}, is important in the proof of Theorems $\ref{27}$, $\ref{114}$ and $\ref{12}$.
\begin{lemma}{\rm( \cite{DLVP})}\label{54}
Let $n$ be a positive integer and $\{a_{t}\}_{t=1}^{\infty}$ be a decreasing sequence of non-negative numbers with
\begin{equation*}
\sum_{t=1}^{+\infty}t^{n-1}a_{t}<+\infty.
\end{equation*}
Then
\begin{equation*}
\lim_{t\to\infty}t^{n}a_{t}=0.
\end{equation*}
\end{lemma}
In order to prove Theorem $\ref{27}$ (\rmnum{2}) and Theorem $\ref{114}$, we need the theory of $\psi$-Dirichlet. The below definition of $\psi$-Dirichlet can be found in \cite[Section 1.2]{DN}.
\begin{definition}
For a decreasing function $\psi:\mathbb{N}\to\mathbb{R}_{\geq0}$, we say that a pair $(A,\boldsymbol{\gamma})\in [0,1)^{m\times n}\times[0,1)^{m}$ is $\psi$-Dirichlet if there exists $\boldsymbol{q}\in\mathbb{Z}^{n}$ such that
\begin{equation*}
\langle A\boldsymbol{q}-\boldsymbol{\gamma}\rangle^{m}<\psi(T)\ {\rm and}\ 1\leq\|\boldsymbol{q}\|^{n}\leq T
\end{equation*}
whenever $T$ is large enough.
\end{definition}
Denote the set of all $\psi$-Dirichlet pairs by $D_{m,n}(\psi)$. For fix $\boldsymbol{\gamma}\in[0,1)^{m}$, let
\begin{equation*}
D_{m,n}^{\boldsymbol{\gamma}}(\psi):=\{A\in[0,1)^{m\times n}: (A,\boldsymbol{\gamma})\in D_{m,n}(\psi)\}.
\end{equation*}
Furthermore, for $\kappa\in[0,+\infty)$, let function $\psi_{\kappa}(T)=T^{-\kappa},\ \forall\ T\in\mathbb{N}$. Denote
\begin{equation*}
\textbf{Sing}_{m,n}^{\boldsymbol{\gamma}}(\kappa):=\bigcap_{\epsilon>0}D_{m,n}^{\boldsymbol{\gamma}}(\epsilon\cdot\psi_{\kappa}),
\end{equation*}
which is the set of all inhomogeneously singular matrices with respect to exponent $\kappa$. When $\boldsymbol{\gamma}=\boldsymbol{0}$ and $\kappa=1$, we call $\textbf{Sing}_{m,n}^{\boldsymbol{0}}(1)$ the set of all singular matrices. The notation of singularity was introduced by Khintchine, first in 1937 in the setting of simultaneous approximation \cite{AK1}, and later in 1948 in the more general setting of matrix approximation \cite{AK2}.
The name singular derives from the fact that $\textbf{Sing}_{m,n}^{\boldsymbol{0}}(1)$ is a Lebesgue null set for all $m,n$, see \cite{AK1}. The Hausdorff dimension of $\textbf{Sing}_{m,n}^{\boldsymbol{0}}(1)$ is a classical and difficult question in Diophantine approximation. In 2024, Das, Fishman, Simmons and Urba\'{n}ski \cite{DFSU} gave the answer.
\begin{lemma}\label{301}{\rm(\cite[Theorem 3.1]{DFSU})}
For all $m,n\in\mathbb{N}$ with $mn>1$, we have $$\dim_{\rm H}\left(\rm{\mathbf{Sing}_{m,n}^{\boldsymbol{0}}(1)}\right)=mn\left(1-\frac{1}{m+n}\right).$$
\end{lemma}
\begin{remark}
When $m=n=1$, by the knowledge of continued fraction, we know that $$\rm{\mathbf{Sing}_{1,1}^{\boldsymbol{0}}(1)}=[0,1)\cap\mathbb{Q}.$$
\end{remark}
Das, Fishman, Simmons and Urba\'{n}ski \cite{DFSU} also studied the Hausdorff dimension of the set of all very singular matrices. More exactly, denote
\begin{equation*}
\textbf{VSing}_{m,n}^{\boldsymbol{0}}:=\bigcup_{\kappa>1}D_{m,n}^{\boldsymbol{0}}(\psi_{\kappa}),
\end{equation*}
they prove that the Hausdorff dimension of $\textbf{VSing}_{m,n}^{\boldsymbol{0}}$ is equal to the Hausdorff dimension of $\textbf{Sing}_{m,n}^{\boldsymbol{0}}(1)$.
\begin{lemma}\label{303}{\rm(\cite[Theorem 3.4]{DFSU})}
For all $m,n\in\mathbb{N}$ with $mn>1$, we have $$\dim_{\rm H}\left(\rm{\mathbf{VSing}_{m,n}^{\boldsymbol{0}}}\right)=mn\left(1-\frac{1}{m+n}\right).$$
\end{lemma}
\begin{remark}
Note that
\begin{equation*}
\bigcup_{\kappa>1}D_{m,n}^{\boldsymbol{0}}(\psi_{\kappa})=\bigcup_{\kappa>1}\textbf{Sing}_{m,n}^{\boldsymbol{0}}(\kappa),
\end{equation*}
in view of Theorem $\ref{303}$, we have
\begin{equation*}
\dim_{\rm H}\left(\bigcup_{\kappa>1}\textbf{Sing}_{m,n}^{\boldsymbol{0}}(\kappa)\right)=mn\left(1-\frac{1}{m+n}\right).
\end{equation*}
\end{remark}
Furthermore, Schleischitz \cite{SCJ} gave a lower bound of the Hausdorff dimension of $\textbf{Sing}_{m,1}^{\boldsymbol{\gamma}}(\kappa)$.
\begin{lemma}\label{19}{\rm(\cite[Theorem 2.2]{SCJ})}
For any $\boldsymbol{\gamma}\in[0,1)^{m}$ and $\kappa\in[0,m)$, we have $$\dim_{\rm H}(\mathbf{Sing}_{m,1}^{\boldsymbol{\gamma}}(\kappa))\geq m\left(\frac{m-\kappa}{m+\kappa}\right)^{2}.$$
\end{lemma}
\begin{remark}
Note that $\mathbf{Sing}_{m,1}^{\boldsymbol{\gamma}}(n\kappa)\times[0,1)^{m\times(n-1)}\subset\mathbf{Sing}_{m,n}^{\boldsymbol{\gamma}}(\kappa)$ for every $\boldsymbol{\gamma}\in[0,1)^{m}$ and $\kappa\in[0,+\infty)$. By Lemma $\ref{19}$ and \cite[Corollary 7.12]{KJF}, we immediately obtain that $$\dim_{\rm H}(\mathbf{Sing}_{m,n}^{\boldsymbol{\gamma}}(\kappa))\geq m(n-1)+m\left(\frac{m-n\kappa}{m+n\kappa}\right)^{2},\ \forall\ \kappa\in\left[0,\frac{m}{n}\right).$$
\end{remark}
The Hausdorff measure of $[0,1)^{m\times n}\setminus D_{m,n}^{\boldsymbol{\gamma}}(\psi)$ was established by T. Kim and W. Kim \cite{KK}.
\begin{lemma}\label{45}{\rm(\cite[Theorem 1.4]{KK})}
Given a decreasing function $\psi$ with $\lim\limits_{T\to+\infty}\psi(T)=0$ and $0\leq s\leq mn$, the $s$-dimensional Hausdorff measure of $[0,1)^{m\times n}\setminus D_{m,n}^{\boldsymbol{\gamma}}(\psi)$ is given by
\begin{equation*}
\mathcal{H}^{s}([0,1)^{m\times n}\setminus D_{m,n}^{\boldsymbol{\gamma}}(\psi))=\begin{cases} 0,  &\text{if}\ \sum\limits_{T=1}^{+\infty}\frac{1}{\psi(T)T^{2}}\left(\frac{T^{\frac{1}{n}}}{\psi(T)^{\frac{1}{m}}}\right)^{mn-s}<+\infty, \\
\mathcal{H}^{s}([0,1)^{m\times n}), &\text{if}\ \sum\limits_{T=1}^{+\infty}\frac{1}{\psi(T)T^{2}}\left(\frac{T^{\frac{1}{n}}}{\psi(T)^{\frac{1}{m}}}\right)^{mn-s}=+\infty, \end{cases}
\end{equation*}
for every $\boldsymbol{\gamma}\in[0,1)^{m}\setminus\{\boldsymbol{0}\}$. Moreover, the convergent case still holds for every $\boldsymbol{\gamma}\in[0,1)^{m}$ and every decreasing function $\psi$ without the assumption $\lim\limits_{T\to+\infty}\psi(T)=0$.
\end{lemma}
\section{Proof of Theorem $\ref{27}$}\label{300}
Firstly, we show that $\Omega^{\boldsymbol{\gamma}}(m,n)$ is contained in $\textbf{Sing}_{m,n}^{\boldsymbol{\gamma}}(1)$, which is crucial in proving Theorem $\ref{27}$ (\rmnum{2}).
\begin{lemma}\label{302}
For all $m,n\in\mathbb{N}$ and any $\boldsymbol{\gamma}\in[0,1)^{m}$, we have
\begin{equation*}
\Omega^{\boldsymbol{\gamma}}(m,n)\subset\mathbf{Sing}_{m,n}^{\boldsymbol{\gamma}}(1).
\end{equation*}
\end{lemma}
\begin{proof}[\rm \textbf{Proof}]
For any $A\in\Omega^{\boldsymbol{\gamma}}(m,n)$, by Lemma $\ref{36}$, we have
\begin{equation*}
\sum_{t=1}^{\infty}t^{n-1}\cdot\left(\min_{1\leq\|\boldsymbol{q}\|\leq t}\langle A\boldsymbol{q}-\boldsymbol{\gamma}\rangle^{m}\right)<+\infty.
\end{equation*}
Since $\min\limits_{1\leq\|\boldsymbol{q}\|\leq t}\langle A\boldsymbol{q}-\boldsymbol{\gamma}\rangle^{m}$ is decreasing, and a straightforward consequence of Lemma $\ref{54}$ is that
\begin{equation*}
\lim_{t\to+\infty}t^{n}\cdot\left(\min_{1\leq\|\boldsymbol{q}\|\leq t}\langle A\boldsymbol{q}-\boldsymbol{\gamma}\rangle^{m}\right)=0.
\end{equation*}
Thus, for any $\epsilon>0$, when $T$ is large enough, we have
\begin{equation*}
\min_{1\leq\|\boldsymbol{q}\|\leq T^{\frac{1}{n}}}\langle A\boldsymbol{q}-\boldsymbol{\gamma}\rangle^{m}<\epsilon\cdot T^{-1}.
\end{equation*}
Therefore
\begin{equation*}
A\in\bigcap_{\epsilon>0}D_{m,n}^{\boldsymbol{\gamma}}(\epsilon\cdot\psi_{1})=\mathbf{Sing}_{m,n}^{\boldsymbol{\gamma}}(1).
\end{equation*}
\textrm{}
\end{proof}
Now we are able to prove Theorem $\ref{27}$.
\begin{proof}[\rm \textbf{Proof of Theorem $\ref{27}$}]
(\rmnum{1}) Note that the conclusion of (\rmnum{1}) is equivalent to
\begin{equation*}
A\in\boldsymbol{\rm Bad}^{\boldsymbol{0}}(m,n) \quad \Leftrightarrow \quad \Omega_{A}(m,n)=\emptyset.
\end{equation*}
Firstly, we prove the ``$\Rightarrow$'' part. Let $A\in\boldsymbol{\rm Bad}^{\boldsymbol{0}}(m,n)$. By Lemma $\ref{36}$, it sufficient to show that for all $\boldsymbol{\gamma}\in[0,1)^{m}$, there exists $l\in\mathbb{N}$, such that
\begin{equation*}
S_{l}(A,\boldsymbol{\gamma})=+\infty.
\end{equation*}
Since $A\in\boldsymbol{\rm Bad}^{\boldsymbol{0}}(m,n)$, there exists a positive constant $c:=c(A)$, satisfies
\begin{equation}\label{58}
\|\boldsymbol{q}\|^{n}\langle A\boldsymbol{q}\rangle^{m}\geq c\ {\rm for\ all}\ \boldsymbol{q}\in\mathbb{Z}^{n}\setminus\{\boldsymbol{0}\}.
\end{equation}
According to the range of $\boldsymbol{\gamma}$, we consider the following two cases. \\
\textbf{Case 1:} $\boldsymbol{\gamma}\in\bigcup\limits_{\boldsymbol{q}\in\mathbb{Z}^{n}}\bigcup\limits_{\boldsymbol{p}\in\mathbb{Z}^{m}}\{A\boldsymbol{q}+\boldsymbol{p}\}$. Then there exists $\boldsymbol{q}_{0}\in\mathbb{Z}^{n}$, $\boldsymbol{p}_{0}\in\mathbb{Z}^{m}$, such that
\begin{equation*}
\boldsymbol{\gamma}=A\boldsymbol{q}_{0}+\boldsymbol{p}_{0}.
\end{equation*}
Let $l=\|\boldsymbol{q}_{0}\|+1$, then we have
\begin{equation*}
\begin{aligned}
S_{l}(A,\boldsymbol{\gamma})&=\sum_{t=l}^{\infty}t^{n-1}\min_{l\leq \|\boldsymbol{q}\|\leq t}\langle A\boldsymbol{q}-\boldsymbol{\gamma}\rangle^{m}
=\sum_{t=l}^{\infty}t^{n-1}\min_{l\leq \|\boldsymbol{q}\|\leq t}\langle A(\boldsymbol{q}-\boldsymbol{q}_{0})\rangle^{m}\\
&\geq \sum_{t=l}^{\infty}t^{n-1}\min_{1\leq \|\boldsymbol{q}-\boldsymbol{q}_{0}\|\leq t+\|\boldsymbol{q}_{0}\|}\langle A(\boldsymbol{q}-\boldsymbol{q}_{0})\rangle^{m}
=\sum_{t=l}^{\infty}t^{n-1}\min_{1\leq \|\boldsymbol{q}\|\leq t+\|\boldsymbol{q}_{0}\|}\langle A\boldsymbol{q}\rangle^{m}\\
&\geq\sum_{t=l}^{\infty}t^{n-1}\min_{1\leq \|\boldsymbol{q}\|\leq t+\|\boldsymbol{q}_{0}\|}c\|\boldsymbol{q}\|^{-n}
=c\sum_{t=l}^{\infty}t^{n-1}(t+\|\boldsymbol{q}_{0}\|)^{-n}=+\infty.
\end{aligned}
\end{equation*}
The first inequality is due to $\|\cdot\|$ satisfies triangle inequality and the second inequality is due to $\eqref{58}$. By Lemma $\ref{36}$, we have $\boldsymbol{\gamma}\notin\Omega_{A}(m,n)$. \\
\textbf{Case 2:}
$\boldsymbol{\gamma}\notin\bigcup\limits_{\boldsymbol{q}\in\mathbb{Z}^{n}}\bigcup\limits_{\boldsymbol{p}\in\mathbb{Z}^{m}}\{A\boldsymbol{q}+\boldsymbol{p}\}$. We claim that $S_{1}(A,\boldsymbol{\gamma})=+\infty$. In fact, note that $\min\limits_{1\leq \|\boldsymbol{q}\|\leq t}\langle A\boldsymbol{q}-\boldsymbol{\gamma}\rangle$ is decreasing, so the limit of $\min\limits_{1\leq \|\boldsymbol{q}\|\leq t}\langle A\boldsymbol{q}-\boldsymbol{\gamma}\rangle$ exists. When
\begin{equation*}
\lim_{t\to\infty}\min_{1\leq \|\boldsymbol{q}\|\leq t}\langle A\boldsymbol{q}-\boldsymbol{\gamma}\rangle\neq 0,
\end{equation*}
since $\lim\limits_{t\to\infty}t^{n-1}\min\limits_{1\leq \|\boldsymbol{q}\|\leq t}\langle A\boldsymbol{q}-\boldsymbol{\gamma}\rangle^{m}\neq0$, we have $S_{1}(A,\boldsymbol{\gamma})=+\infty$. The left is just the case that
\begin{equation}\label{61}
\lim_{t\to\infty}\min_{1\leq \|\boldsymbol{q}\|\leq t}\langle A\boldsymbol{q}-\boldsymbol{\gamma}\rangle=0.
\end{equation}
Since $\langle A\boldsymbol{q}-\boldsymbol{\gamma}\rangle>0$ for all $\boldsymbol{q}\in\mathbb{Z}^{n}$, we can construct a infinite sequence $\{t_{k}\}_{k=1}^{\infty}$ satisfies for any $k\geq1$ and $t<t_{k+1}$,
\begin{equation*}
\min_{1\leq\|\boldsymbol{q}\|\leq t}\langle A\boldsymbol{q}-\boldsymbol{\gamma}\rangle\geq\min_{1\leq\|\boldsymbol{q}\|\leq t_{k}}\langle A\boldsymbol{q}-\boldsymbol{\gamma}\rangle.
\end{equation*}
More exactly, take $t_{1}=1$, for any $k\geq1$, suppose the positive integers $t_{1},t_{2},\cdots,t_{k}$ have been determined, let
\begin{equation*}
t_{k+1}=\min\left\{t\in\mathbb{N}:t>t_{k}\ {\rm and}\ \min_{1\leq\|\boldsymbol{q}\|\leq t}\langle A\boldsymbol{q}-\boldsymbol{\gamma}\rangle<\min_{1\leq\|\boldsymbol{q}\|\leq t_{k}}\langle A\boldsymbol{q}-\boldsymbol{\gamma}\rangle\right\}.
\end{equation*}
What is more, by the definition of $\{t_{k}\}_{k=1}^{\infty}$, for every $k\geq1$, we can choose $\boldsymbol{q}_{k}\in\mathbb{Z}^{n}$ with
\begin{equation*}
\|\boldsymbol{q}_{k}\|=t_{k}
\end{equation*}
and
\begin{equation*}
\langle A\boldsymbol{q}_{k}-\boldsymbol{\gamma}\rangle=\min_{1\leq \|\boldsymbol{q}\|\leq t_{k}}\langle A\boldsymbol{q}-\boldsymbol{\gamma}\rangle.
\end{equation*}
Therefore
\begin{equation}\label{62}
\begin{aligned}
S_{1}(A,\boldsymbol{\gamma})=\sum_{t=1}^{+\infty}t^{n-1}\min_{1\leq \|\boldsymbol{q}\|\leq t}\langle A\boldsymbol{q}-\boldsymbol{\gamma}\rangle^{m}&=\sum_{k=1}^{+\infty}\sum_{t=t_{k}}^{t_{k+1}-1}t^{n-1}\min_{1\leq \|\boldsymbol{q}\|\leq t}\langle A\boldsymbol{q}-\boldsymbol{\gamma}\rangle^{m}\\
&=\sum_{k=1}^{+\infty}\left(\min_{1\leq \|\boldsymbol{q}\|\leq t_{k}}\langle A\boldsymbol{q}-\boldsymbol{\gamma}\rangle^{m}\sum_{t=t_{k}}^{t_{k+1}-1}t^{n-1}\right)\\
&=\sum_{k=1}^{+\infty}\left(\langle A\boldsymbol{q}_{k}-\boldsymbol{\gamma}\rangle^{m}\sum_{t=\|\boldsymbol{q}_{k}\|}^{ \|\boldsymbol{q}_{k+1}\|-1}t^{n-1}\right).
\end{aligned}
\end{equation}
Since $\langle\cdot\rangle$ satisfies the triangle inequality, we have
\begin{equation}\label{63}
\begin{aligned}
2\langle A\boldsymbol{q}_{k}-\boldsymbol{\gamma}\rangle>\langle A\boldsymbol{q}_{k+1}-\boldsymbol{\gamma}\rangle+\langle A\boldsymbol{q}_{k}-\boldsymbol{\gamma}\rangle&\geq\langle A\boldsymbol{q}_{k+1}-\boldsymbol{\gamma}-(A\boldsymbol{q}_{k}-\boldsymbol{\gamma})\rangle\\&=\langle A(\boldsymbol{q}_{k+1}-\boldsymbol{q}_{k})\rangle.
\end{aligned}
\end{equation}
The first inequality of $\eqref{63}$ is due to the choice of $\{\boldsymbol{q}_{k}\}_{k=1}^{\infty}$. Applying Lemma $\ref{44}$ to $\epsilon=2\langle A\boldsymbol{q}_{k}-\boldsymbol{\gamma}\rangle$, we obtain that
\begin{equation*}
\|\boldsymbol{q}_{k+1}-\boldsymbol{q}_{k}\|\geq\min\{\|\boldsymbol{q}\|: \boldsymbol{q}\in\mathbb{Z}^{n}\setminus\{\boldsymbol{0}\},\ \langle A\boldsymbol{q}\rangle<2\langle A\boldsymbol{q}_{k}-\boldsymbol{\gamma}\rangle\}\geq\left(\frac{c}{2^{m}\langle A\boldsymbol{q}_{k}-\boldsymbol{\gamma}\rangle^{m}}\right)^{\frac{1}{n}}.
\end{equation*}
It follows that
\begin{equation}\label{64}
\langle A\boldsymbol{q}_{k}-\boldsymbol{\gamma}\rangle^{m}\geq c\cdot2^{-m}\|\boldsymbol{q}_{k+1}-\boldsymbol{q}_{k}\|^{-n}.
\end{equation}
Combining $\eqref{62}$ and $\eqref{64}$, we have
\begin{equation*}
S_{1}(A,\boldsymbol{\gamma})\geq c\cdot2^{-m}\sum_{k=1}^{+\infty}\left(\|\boldsymbol{q}_{k+1}-\boldsymbol{q}_{k}\|^{-n}\sum_{t=\|\boldsymbol{q}_{k}\|}^{\|\boldsymbol{q}_{k+1}\| -1}t^{n-1}\right).
\end{equation*}
It suffices to prove that
\begin{equation*}
\sum_{k=1}^{+\infty}\left(\|\boldsymbol{q}_{k+1}-\boldsymbol{q}_{k}\|^{-n}\sum_{t=\|\boldsymbol{q}_{k}\|}^{\|\boldsymbol{q}_{k+1}\|-1}t^{n-1}\right)=+\infty.
\end{equation*}
We will finish it by two cases according to $\limsup\limits_{k\to+\infty}\|\boldsymbol{q}_{k+1}\|\cdot\|\boldsymbol{q}_{k}\|^{-1}<+\infty$ or $=+\infty$.\\
(1) If $\limsup\limits_{k\to+\infty}\|\boldsymbol{q}_{k+1}\|\cdot\|\boldsymbol{q}_{k}\|^{-1}<+\infty$, then there exists $1<c_{1}<+\infty$, such that
\begin{equation}\label{66}
\|\boldsymbol{q}_{k+1}\|\cdot\|\boldsymbol{q}_{k}\|^{-1}\leq c_{1},\ \forall\ k\geq1.
\end{equation}
It follows that
\begin{equation*}
\begin{aligned}
\sum_{k=1}^{+\infty}\left(\|\boldsymbol{q}_{k+1}-\boldsymbol{q}_{k}\|^{-n}\sum_{t=\|\boldsymbol{q}_{k}\|}^{\|\boldsymbol{q}_{k+1}\|-1}t^{n-1}\right)
&\geq \sum_{k=1}^{+\infty}\left(2^{-n}\|\boldsymbol{q}_{k+1}\|^{-n}\sum_{t=\|\boldsymbol{q}_{k}\|}^{\|\boldsymbol{q}_{k+1}\|-1}t^{n-1}\right) \\
&\geq 2^{-n}\sum_{k=1}^{+\infty}\left(\|\boldsymbol{q}_{k+1}\|^{-n}\|\boldsymbol{q}_{k}\|^{n}\sum_{t=\|\boldsymbol{q}_{k}\|}^{\|\boldsymbol{q}_{k+1}\|-1}t^{-1}\right)\\
&\geq 2^{-n}c_{1}^{-n}\sum_{t=1}^{+\infty}t^{-1}=+\infty.
\end{aligned}
\end{equation*}
The first inequality is due to $\|\cdot\|$ satisfies the triangle inequality. The third inequality is due to $\eqref{66}$. \\
(2) If $\limsup\limits_{k\to+\infty}\|\boldsymbol{q}_{k+1}\|\cdot\|\boldsymbol{q}_{k}\|^{-1}=+\infty$, since the function $f(x)=x^{n-1}$ is increasing in the interval $[0,+\infty)$, we have
\begin{equation*}
\begin{aligned}
&\quad\sum_{k=1}^{+\infty}\left(\|\boldsymbol{q}_{k+1}-\boldsymbol{q}_{k}\|^{-n}\sum_{t=\|\boldsymbol{q}_{k}\|}^{\|\boldsymbol{q}_{k+1}\|-1}t^{n-1}\right)\\
&\geq \sum_{k=1}^{+\infty}\left(\|\boldsymbol{q}_{k+1}-\boldsymbol{q}_{k}\|^{-n}\int_{\|\boldsymbol{q}_{k}\|-1}^{\|\boldsymbol{q}_{k+1}\|-1}x^{n-1}\mathrm{d}x\right) \\ &=n^{-1}\sum_{k=1}^{+\infty}\|\boldsymbol{q}_{k+1}-\boldsymbol{q}_{k}\|^{-n}\left[(\|\boldsymbol{q}_{k+1}\|-1)^{n}-(\|\boldsymbol{q}_{k}\|-1)^{n}\right]\\
&\geq n^{-1}2^{-n}\sum_{k=1}^{+\infty}\|\boldsymbol{q}_{k+1}\|^{-n}\left[(\|\boldsymbol{q}_{k+1}\|-1)^{n}-(\|\boldsymbol{q}_{k}\|-1)^{n}\right].
\end{aligned}
\end{equation*}
Since $\limsup\limits_{k\to+\infty}\frac{\|\boldsymbol{q}_{k+1}\|}{\|\boldsymbol{q}_{k}\|}=+\infty$ and $\lim\limits_{k\to+\infty}\|\boldsymbol{q}_{k}\|=+\infty$, we have
\begin{equation*}
\limsup_{k\to+\infty}\|\boldsymbol{q}_{k+1}\|^{-n}\left[(\|\boldsymbol{q}_{k+1}\|-1)^{n}-(\|\boldsymbol{q}_{k}\|-1)^{n}\right]=1.
\end{equation*}
It follows that
\begin{equation*}
\sum_{k=1}^{+\infty}\|\boldsymbol{q}_{k+1}\|^{-n}\left[(\|\boldsymbol{q}_{k+1}\|-1)^{n}-(\|\boldsymbol{q}_{k}\|-1)^{n}\right]=+\infty.
\end{equation*}
Therefore
\begin{equation*}
\sum_{k=1}^{+\infty}\left(\|\boldsymbol{q}_{k+1}-\boldsymbol{q}_{k}\|^{-n}\sum_{t=\|\boldsymbol{q}_{k}\|}^{\|\boldsymbol{q}_{k+1}\|-1}t^{n-1}\right)=+\infty.
\end{equation*}
Secondly, we prove the ``$\Leftarrow$'' part by contradiction. Let $A\in[0,1)^{m\times n}$ with $\Omega_{A}(m,n)=\emptyset$. Suppose that $A\notin{\boldsymbol{\rm Bad}}^{\boldsymbol{0}}(m,n)$, then we have
\begin{equation}\label{72}
\liminf_{\boldsymbol{q}\in\mathbb{Z}^{n},\|\boldsymbol{q}\|\to+\infty}\|\boldsymbol{q}\|^{n}\langle A\boldsymbol{q}\rangle^{m}=0.
\end{equation}
If there exists $\boldsymbol{q}\in\mathbb{Z}^{n}\setminus\{\boldsymbol{0}\}$ such that $\langle A\boldsymbol{q}\rangle=0$, then for any $\psi\in\mathcal{D}$, we have $\langle A\boldsymbol{q}\rangle=0<\psi(\|\boldsymbol{q}\|)$ for infinitely many $\boldsymbol{q}\in\mathbb{Z}^{n}$. Thus, $\Omega_{A}(m,n)\neq\emptyset$. So $\langle A\boldsymbol{q}\rangle\neq0$ for every $\boldsymbol{q}\in\mathbb{Z}^{n}\setminus\{\boldsymbol{0}\}$. By $\eqref{72}$, there exists a sequence $\{\boldsymbol{q}_{k}\}_{k=1}^{\infty}\subset\mathbb{Z}^{n}\setminus\{\boldsymbol{0}\}$, satisfies
\begin{equation*}
\|\boldsymbol{q}_{k+1}\|\geq(k+1)\cdot(\|\boldsymbol{q}_{1}\|+\cdots+\|\boldsymbol{q}_{k}\|)+1
\end{equation*}
and
\begin{equation*}
\|\boldsymbol{q}_{k+1}\|^{n}\langle A\boldsymbol{q}_{k+1}\rangle^{m}<2^{-m}\|\boldsymbol{q}_{k}\|^{n}\langle A\boldsymbol{q}_{k}\rangle^{m}.
\end{equation*}
Thus, we have
\begin{equation}\label{73}
\sum_{k=1}^{+\infty}\|\boldsymbol{q}_{k}\|^{n}\langle A\boldsymbol{q}_{k}\rangle^{m}<+\infty
\end{equation}
and
\begin{equation}\label{74}
\sum_{k=K}^{+\infty}\langle A\boldsymbol{q}_{k}\rangle\leq2\langle A\boldsymbol{q}_{K}\rangle\ {\rm for\ any}\ K\in\mathbb{N}.
\end{equation}
Define
\begin{equation}\label{81}
\boldsymbol{\gamma}=\sum_{k=1}^{+\infty}(A\boldsymbol{q}_{k}-\boldsymbol{p}_{k}),
\end{equation}
where $\boldsymbol{p}_{k}\in\mathbb{Z}^{m}$ is such that $\|A\boldsymbol{q}_{k}-\boldsymbol{p}_{k}\|=\langle A\boldsymbol{q}_{k}\rangle$.
Let us write $\boldsymbol{\gamma}$ as $(\gamma_{1},\cdots,\gamma_{m})$ and denote
\begin{equation*}
\boldsymbol{\gamma}'=(\{\gamma_{1}\},\cdots,\{\gamma_{m}\}),
\end{equation*}
where $\{\gamma_{i}\}$ represents the fractional part of $\gamma_{i}$. For each $K\in\mathbb{N}$, denote
\begin{equation}\label{75}
N_{K}=\sum_{k=1}^{K}\|\boldsymbol{q}_{k}\|.
\end{equation}
For any $l\in\mathbb{N}$, choose $K\in\mathbb{N}$ with $N_{K}\geq l$. We claim that $S_{N_{K}}(A,\boldsymbol{\gamma}')<+\infty$, and this will imply, by Lemma $\ref{43}$, $S_{l}(A,\boldsymbol{\gamma}')<+\infty$. Indeed, note that
\begin{equation}\label{76}
\begin{aligned}
&\quad\sum_{t=N_{K+1}+1}^{+\infty}t^{n-1}\min_{N_{K}\leq\|\boldsymbol{q}\|\leq t}\langle A\boldsymbol{q}-\boldsymbol{\gamma'}\rangle^{m} \\
&=\sum_{t=N_{K+1}+1}^{+\infty}t^{n-1}\min_{N_{K}\leq\|\boldsymbol{q}\|\leq t}\langle A\boldsymbol{q}-\boldsymbol{\gamma}\rangle^{m} \\
&=\sum_{i=K+1}^{+\infty}\sum_{t=N_{i}+1}^{N_{i+1}}t^{n-1}\min_{N_{K}\leq\|\boldsymbol{q}\|\leq t}\langle A\boldsymbol{q}-\boldsymbol{\gamma}\rangle^{m}.
\end{aligned}
\end{equation}
By $\eqref{75}$ and the choice of $\{\boldsymbol{q}_{k}\}_{k=1}^{\infty}$, for every $i\geq K+1$ and $N_{i}+1\leq t\leq N_{i+1}$, we have
\begin{equation}\label{77}
N_{K}\leq \|\boldsymbol{q}_{1}+\cdots+\boldsymbol{q}_{i}\|\leq t.
\end{equation}
In view of $\eqref{81}$, $\eqref{76}$ and $\eqref{77}$, we obtain that
\begin{equation}\label{82}
\begin{aligned}
&\quad\sum_{t=N_{K+1}+1}^{+\infty}t^{n-1}\min_{N_{K}\leq\|\boldsymbol{q}\|\leq t}\langle A\boldsymbol{q}-\boldsymbol{\gamma'}\rangle^{m} \\
&\leq \sum_{i=K+1}^{+\infty}\sum_{t=N_{i}+1}^{N_{i+1}}t^{n-1}\langle A(\boldsymbol{q}_{1}+\cdots+\boldsymbol{q}_{i})-\boldsymbol{\gamma}\rangle^{m} \\
&=\sum_{i=K+1}^{+\infty}\left(\left\langle \sum_{k=i+1}^{+\infty}(A\boldsymbol{q}_{k}-\boldsymbol{p}_{k})\right\rangle^{m}\sum_{t=N_{i}+1}^{N_{i+1}}t^{n-1}\right).
\end{aligned}
\end{equation}
Note that
\begin{equation}\label{83}
\begin{aligned}
\left\langle\sum_{k=i+1}^{+\infty}(A\boldsymbol{q}_{k}-\boldsymbol{p}_{k})\right\rangle&\leq \left\|\sum_{k=i+1}^{+\infty}(A\boldsymbol{q}_{k}-\boldsymbol{p}_{k})\right\|\leq \sum_{k=i+1}^{+\infty}\|A\boldsymbol{q}_{k}-\boldsymbol{p}_{k}\|\\ &=\sum_{k=i+1}^{+\infty}\langle A\boldsymbol{q}_{k}\rangle
\leq 2\langle A\boldsymbol{q}_{i+1}\rangle,
\end{aligned}
\end{equation}
the last inequality is due to $\eqref{74}$. What is more, by the choice of $\{\boldsymbol{q}_{k}\}_{k=1}^{\infty}$ and $\eqref{75}$, we have
\begin{equation*}
\lim_{i\to+\infty}\frac{(N_{i+1}+1)^{n}-(N_{i}+1)^{n}}{\|\boldsymbol{q}_{i+1}\|^{n}}=1.
\end{equation*}
It follows that
\begin{equation}\label{84}
\sum_{t=N_{i}+1}^{N_{i+1}}t^{n-1}\leq\int_{N_{i}+1}^{N_{i+1}+1}x^{n-1}\mathrm{d}x =\frac{1}{n}[(N_{i+1}+1)^{n}-(N_{i}+1)^{n}]
\leq c_{2}\|\boldsymbol{q}_{i+1}\|^{n},
\end{equation}
where $c_{2}>0$ is a constant only depends on $A$. By $\eqref{82}$, $\eqref{83}$ and $\eqref{84}$, we have
\begin{equation*}
\sum_{t=N_{K+1}+1}^{+\infty}t^{n-1}\min_{N_{K}\leq\|\boldsymbol{q}\|\leq t}\langle A\boldsymbol{q}-\boldsymbol{\gamma'}\rangle^{m}
\leq 2^{m}c_{2}\sum_{i=K+1}^{+\infty}\|\boldsymbol{q}_{i+1}\|^{n}\langle A\boldsymbol{q}_{i+1}\rangle^{m}<+\infty.
\end{equation*}
The last inequality is due to $\eqref{73}$. Hence $S_{l}(A,\boldsymbol{\gamma'})<+\infty$ for any $l\geq1$. By Lemma $\ref{36}$, we have $\boldsymbol{\gamma'}\in\Omega_{A}(m,n)$, which contradicts with $\Omega_{A}(m,n)=\emptyset$. Hence, $A\in\boldsymbol{\rm Bad}^{\boldsymbol{0}}(m,n)$.

(\rmnum{2}) If $\boldsymbol{\gamma}=\boldsymbol{0}$, by Khintchine's result \cite{AK1}, $$\mu_{mn}(\textbf{Sing}_{m,n}^{\boldsymbol{0}}(1))=0.$$
Combining this and Lemma $\ref{302}$, we have
\begin{equation*}
\mu_{mn}(\Omega^{\boldsymbol{0}}(m,n))=0.
\end{equation*}
If $\boldsymbol{\gamma}\in[0,1)^{m}\setminus\{\boldsymbol{0}\}$, it follows from Lemma $\ref{45}$ that $\mu_{mn}(D_{m,n}^{\boldsymbol{\gamma}}(\psi_{1}))=0$. It follows that $$\mu_{mn}(\mathbf{Sing}_{m,n}^{\boldsymbol{\gamma}}(1))=0.$$ This together with Lemma $\ref{302}$ shows that
\begin{equation*}
\mu_{mn}(\Omega^{\boldsymbol{\gamma}}(m,n))=0.
\end{equation*}
(\rmnum{3}) If $A\in\boldsymbol{\rm Bad}^{\boldsymbol{0}}(m,n)$, by Theorem $\ref{27}$ (\rmnum{1}), we know that
\begin{equation*}
\Omega_{A}(m,n)=\emptyset.
\end{equation*}
If $A\notin\boldsymbol{\rm Bad}^{\boldsymbol{0}}(m,n)$, by \cite[Lemma 12]{JK}, there exists a $\psi_{0}\in\mathcal{D}$, satisfies
\begin{equation*}
\mu_{m}(W_{m,n,A}(\psi_{0}))=0.
\end{equation*}
Since $\Omega_{A}(m,n)\subset W_{m,n,A}(\psi_{0})$, we have $\mu_{m}(\Omega_{A}(m,n))=0$.
\textrm{}
\end{proof}

\section{Proof of Theorems $\ref{114}$, $\ref{120}$ and Corollary $\ref{115}$}\label{400}
\subsection{Proof of Theorem $\ref{114}$ and Corollary $\ref{115}$}
\begin{proof}[\rm \textbf{Proof of Theorem $\ref{114}$}]
Let
\begin{equation*}
\Gamma^{\boldsymbol{\gamma}}(m,n)=\left\{A\in[0,1)^{m\times n}: S_{1}(A,\boldsymbol{\gamma})<\infty\right\}.
\end{equation*}
Recall that $S_{1}(A,\boldsymbol{\gamma})=\sum\limits_{t=1}^{+\infty}t^{n-1}\cdot\min\limits_{1\leq\|\boldsymbol{q}\|\leq t}\langle A\boldsymbol{q}-\boldsymbol{\gamma}\rangle^{m}$. Furthermore, we denote
\begin{equation*}
\mathcal{R}^{\boldsymbol{\gamma}}(m,n)=\{A\in[0,1)^{m\times n}: A\boldsymbol{q}-\boldsymbol{\gamma}\in\mathbb{Z}^{m}\ {\rm for\ some}\ \boldsymbol{q}\in\mathbb{Z}^{n}\setminus\{\boldsymbol{0}\}\},
\end{equation*}
which is the union of countably many hyperplanes of dimension $m(n-1)$. In view of Corollary $\ref{41}$, we have
\begin{equation}\label{85}
\Gamma^{\boldsymbol{\gamma}}(m,n)\setminus\mathcal{R}^{\boldsymbol{\gamma}}(m,n) =\Omega^{\boldsymbol{\gamma}}(m,n)\setminus\mathcal{R}^{\boldsymbol{\gamma}}(m,n).
\end{equation}
Note that
\begin{equation}\label{86}
\bigcup_{\kappa>1}D_{m,n}^{\boldsymbol{\gamma}}(\psi_{\kappa})\subset\Gamma^{\boldsymbol{\gamma}}(m,n).
\end{equation}
Indeed, for any $A\in\bigcup\limits_{\kappa>1}D_{m,n}^{\boldsymbol{\gamma}}(\psi_{\kappa})$, there exists $\kappa>1$ such that $$A\in D_{m,n}^{\boldsymbol{\gamma}}(\psi_{\kappa}).$$ Thus, for all positive integer $t$ large enough, there exists $\boldsymbol{q}\in\mathbb{Z}^{n}$ with $1\leq\|q\|\leq t$, such that $$\langle A\boldsymbol{q}-\boldsymbol{\gamma}\rangle^{m}<(t^{n})^{-\kappa}=t^{-n\kappa}.$$ That is, $$\min\limits_{1\leq\|\boldsymbol{q}\|\leq t}\langle A\boldsymbol{q}-\boldsymbol{\gamma}\rangle^{m}<t^{-n\kappa}.$$ Since $n\kappa>n$, we have that the series $\sum\limits_{t=1}^{+\infty}t^{n-1-n\kappa}$ converges. It follows that the series $S_{1}(A,\boldsymbol{\gamma})=\sum\limits_{t=1}^{+\infty}t^{n-1}\cdot\min\limits_{1\leq\|\boldsymbol{q}\|\leq t}\langle A\boldsymbol{q}-\boldsymbol{\gamma}\rangle^{m}$ converges. Therefore $A\in\Gamma^{\boldsymbol{\gamma}}(m,n)$.\\
(\rmnum{1}) Firstly, by Lemmas $\ref{301}$ and $\ref{302}$, we have
\begin{equation*}
\dim_{\rm H}(\Omega^{\boldsymbol{0}}(m,n))\leq mn\left(1-\frac{1}{m+n}\right).
\end{equation*}
Secondly, the combination of Lemma $\ref{303}$ and $\eqref{86}$ gives
\begin{equation*}
\dim_{\rm H}(\Gamma^{\boldsymbol{0}}(m,n))\geq mn\left(1-\frac{1}{m+n}\right)>m(n-1).
\end{equation*}
Thus, $$\dim_{\rm H}(\Gamma^{\boldsymbol{0}}(m,n)\setminus\mathcal{R}^{\boldsymbol{0}}(m,n))=\dim_{\rm H}(\Gamma^{\boldsymbol{0}}(m,n))\geq mn\left(1-\frac{1}{m+n}\right).$$
It follows from $\eqref{85}$ that $$\dim_{\rm H}(\Omega^{\boldsymbol{0}}(m,n))=\dim_{\rm H}(\Omega^{\boldsymbol{ 0}}(m,n)\setminus\mathcal{R}^{\boldsymbol{0}}(m,n))\geq mn\left(1-\frac{1}{m+n}\right).$$
Therefore $$\dim_{\rm H}(\Omega^{\boldsymbol{0}}(m,n))=mn\left(1-\frac{1}{m+n}\right).$$
(\rmnum{2})
It follows from Lemma $\ref{19}$ that
\begin{equation*}
\dim_{\rm H}(\mathbf{Sing}_{m,n}^{\boldsymbol{\gamma}}(\kappa))\geq m(n-1)+m\left(\frac{m-n\kappa}{m+n\kappa}\right)^{2},\ \forall\ \kappa\in\left[0,\frac{m}{n}\right).
\end{equation*}
This together with $\eqref{86}$ gives
\begin{equation*}
\dim_{\rm H}(\Gamma^{\boldsymbol{\gamma}}(m,n))\geq m(n-1)+m\left(\frac{m-n\kappa}{m+n\kappa}\right)^{2},\   \forall\ \kappa\in\left(1,\frac{m}{n} \right).
\end{equation*}
Letting $\kappa\to1$, we obtain that
\begin{equation*}
\dim_{\rm H}(\Gamma^{\boldsymbol{\gamma}}(m,n))\geq m(n-1)+m\left(\frac{m-n}{m+n}\right)^{2}>m(n-1).
\end{equation*}
Thus,
\begin{equation}\label{87}
\dim_{\rm H}(\Gamma^{\boldsymbol{\gamma}}(m,n)\setminus\mathcal{R}^{\boldsymbol{\gamma}}(m,n))\geq m(n-1)+m\left(\frac{m-n}{m+n}\right)^{2}>m(n-1).
\end{equation}
The combination of $\eqref{85}$ and $\eqref{87}$ gives
\begin{equation*}
\dim_{\rm H}(\Omega^{\boldsymbol{\gamma}}(m,n))=\dim_{\rm H}(\Omega^{\boldsymbol{\gamma}}(m,n)\setminus\mathcal{R}^{\boldsymbol{\gamma}}(m,n))\geq m(n-1)+m\left(\frac{m-n}{m+n}\right)^{2}.
\end{equation*}
\textrm{}
\end{proof}

\begin{proof}[\rm \textbf{Proof of Corollary $\ref{115}$}]
In view of Theorem $\ref{114}$ (\rmnum{2}), we only need to show that $\Omega^{\boldsymbol{\gamma}}(m,n)$ is not an empty set when $n\geq2$. Let
\begin{equation*}
A=\begin{pmatrix}
\gamma_{1} & 0 &\cdots & 0\\
\gamma_{2} & 0 &\cdots & 0\\
\vdots & \vdots & \ddots & \vdots \\
\gamma_{m} & 0 &\cdots & 0
\end{pmatrix}.
\end{equation*}
Then
\begin{equation*}
A(1,q,0,\cdots,0)^{T}=\boldsymbol{\gamma}
\end{equation*}
for all $q\in\mathbb{N}$, where $(1,q,0,\cdots,0)^{T}$ denotes the transpose of $(1,q,0,\cdots,0)$. Hence $\langle A\boldsymbol{q}-\boldsymbol{\gamma}\rangle=0$ for infinitely many $\boldsymbol{q}\in\mathbb{Z}^{n}$. Thus, for any $\psi\in\mathcal{D}$, we have $\langle A\boldsymbol{q}-\boldsymbol{\gamma}\rangle=0<\psi(\|\boldsymbol{q}\|)$ for infinitely many $\boldsymbol{q}\in\mathbb{Z}^{n}$. It follows that $A\in\Omega^{\boldsymbol{\gamma}}(m,n)$. Therefore $\Omega^{\boldsymbol{\gamma}}(m,n)\neq\emptyset$.
\textrm{}
\end{proof}
\subsection{Proof of Theorem $\ref{120}$}
In order to prove Theorem $\ref{120}$, we need to introduce the following notations. Given $\alpha\in[0,1)\setminus\mathbb{Q}$ and $\tau>0$, we denote $$\mathcal{U}_{\tau}[\alpha]:=\{\gamma\in[0,1):{\rm for\ all\ large}\ Q, 1\leq\exists\ q\leq Q\ {\rm such\ that}\ \langle q\alpha-\gamma\rangle<Q^{-\tau}\}.$$ Let $q_{k}=q_{k}(\alpha)$ be the denominator of the $k$-th convergent of the continued fraction of $\alpha$. Recall that $w(\alpha)$ is the irrationality exponent of $\alpha$, that is, $$w(\alpha):=\sup\{s>0:\langle q\alpha\rangle<q^{-s}\ {\rm for\ i.m.}\ q\in\mathbb{N}\}.$$
The following lemma gives a description of $\dim_{\rm H}(\mathcal{U}_{\tau}[\alpha])$ when $w(\alpha)>1$.
\begin{lemma}\label{92}{\rm\cite[Theorem 1]{DL}}
Let $\alpha\in[0,1)\setminus\mathbb{Q}$ with $w(\alpha)>1$. Then
\begin{equation*}
\dim_{\rm H}(\mathcal{U}_{\tau}[\alpha])=\begin{cases} 1, &\text{if}\ \tau<\frac{1}{w(\alpha)},\\ \liminf\limits_{k\to+\infty}\frac{\log\left(n_{k}^{1+\frac{1}{\tau}}\prod_{j=1}^{k-1}n_{j}^{\frac{1}{\tau}}\langle n_{j}\alpha\rangle\right)}{\log(n_{k}\langle n_{k}\alpha\rangle^{-1})}, &\text{if}\ \frac{1}{w(\alpha)}<\tau<1,\\ \liminf\limits_{k\to+\infty}\frac{-\log\left(\prod_{j=1}^{k-1}n_{j}\langle n_{j}\alpha\rangle^{\frac{1}{\tau}}\right)}{\log(n_{k}\langle n_{k}\alpha\rangle^{-1})}, &\text{if}\ 1<\tau<w(\alpha), \\ 0, &\text{if}\ \tau>w(\alpha), \end{cases}
\end{equation*}
where $(n_{k})_{k=1}^{\infty}$ is the maximal subsequence of $(q_{k})_{k=1}^{\infty}$ such that
\begin{equation*}
\begin{cases} n_{k}\langle n_{k}\alpha\rangle^{\tau}<1, &\text{if}\ \frac{1}{w(\alpha)}<\tau<1,\\ n_{k}^{\tau}\langle n_{k}\alpha\rangle<2, &\text{if}\ 1<\tau<w(\alpha). \end{cases}
\end{equation*}
\end{lemma}
The formula for $\dim_{\rm H}(\mathcal{U}_{\tau}[\alpha])$ with $w(\alpha)>1$ in the Lemma $\ref{92}$ is a little complicated, and we do not know the information of $\dim_{\rm H}(\mathcal{U}_{\tau}[\alpha])$ when $\tau=1,\frac{1}{w(\alpha)},w(\alpha)$ from Lemma $\ref{92}$. On the basis of Lemma $\ref{92}$, under the assumption $w(\alpha)>1$, the following lemma gives a estimate for $\dim_{\rm H}(\mathcal{U}_{\tau}[\alpha])$ when $\tau\in[\frac{1}{w(\alpha)},w(\alpha)]$.
\begin{lemma}\label{93}{\rm(\cite[Theorem 3]{DL})}
For any $\alpha\in[0,1)\setminus\mathbb{Q}$ with $w(\alpha)=w>1$, we have
\begin{equation*}
\frac{\frac{w}{\tau}-1}{w^{2}-1}\leq\dim_{\rm H}(\mathcal{U}_{\tau}[\alpha])\leq \frac{\frac{1}{\tau}+1}{w+1},\ \text{if}\ \frac{1}{w }\leq\tau\leq1,
\end{equation*}
\begin{equation*}
0\leq \dim_{\rm H}(\mathcal{U}_{\tau}[\alpha])\leq \frac{\frac{w}{\tau}-1}{w^{2}-1},\ \text{if}\ 1<\tau\leq w.
\end{equation*}
\end{lemma}
The following lemma is crucial in the proof of Theorem $\ref{120}$.
\begin{lemma}\label{304}
Let $\alpha\in[0,1)\setminus\mathbb{Q}$ with $$\liminf_{k\to\infty}\frac{\log q_{k+1}}{\log q_{k}}>1.$$ Then
\begin{equation*}
\dim_{\rm H}\left(\bigcup_{\tau>1}\mathcal{U}_{\tau}[\alpha]\right)=\dim_{\rm H}(\mathcal{U}_{1}[\alpha])=\frac{1}{w(\alpha)+1}.
\end{equation*}
\end{lemma}
\begin{proof}[\rm \textbf{Proof}]
We will use the following two important properties of continued fraction. \\
Property 1:
\begin{equation}\label{101}
\frac{1}{2q_{k+1}}<\frac{1}{q_{k+1}+q_{k}}<\langle q_{k}\alpha\rangle<\frac{1}{q_{k+1}}, \quad\ \forall\ k\in\mathbb{N}.
\end{equation}
Property 2:
\begin{equation*}
\langle q_{k}\alpha\rangle\leq\langle n\alpha\rangle,\ \forall\ k\in\mathbb{N}\ {\rm and}\ \forall\ 1\leq n<q_{k+1}.
\end{equation*}
The proof of properties 1 and 2 can be found in \cite[Chapters \Rmnum{1} and \Rmnum{2}]{AK3}. By the definition of $w(\alpha)$ and the above two properties, we immediately obtain that
\begin{equation}\label{103}
w(\alpha)=\limsup_{k\to+\infty}\frac{\log q_{k+1}}{\log q_{k}}.
\end{equation}
Then we will prove the upper and lower bound in Lemma $\ref{304}$ coincide. We consider two cases. If $w(\alpha)=+\infty$, it follows from Lemma $\ref{93}$ that
\begin{equation*}
\dim_{\rm H}(\mathcal{U}_{1}[\alpha])\leq\frac{2}{w(\alpha)+1}=0.
\end{equation*}
Therefore
\begin{equation*}
\dim_{\rm H}\left(\bigcup_{\tau>1}\mathcal{U}_{\tau}[\alpha]\right)=\dim_{\rm H}(\mathcal{U}_{1}[\alpha])=0.
\end{equation*}
If $w(\alpha)<+\infty$, by $\eqref{103}$, we know that $\liminf\limits_{k\to+\infty}\frac{\log q_{k+1}}{\log q_{k}}<+\infty$. For simplicity's sake, we denote $w'=\liminf\limits_{k\to+\infty}\frac{\log q_{k+1}}{\log q_{k}}$. Fix $\epsilon>0$ with $w'-\epsilon>1$, since $\liminf\limits_{k\to+\infty}\frac{\log q_{k+1}}{\log q_{k}}>w'-\epsilon$, we have
\begin{equation*}
\frac{\log q_{k+1}}{\log q_{k}}>w'-\epsilon, \quad\ \forall\ k\gg1.
\end{equation*}
Here and throughout, ``$k\gg1$'' stands for ``$k$ large enough''. Therefore,
\begin{equation}\label{111}
(w'-\epsilon)^{-1}\log q_{k+1}>\log q_{k}, \quad\ \forall\ k\gg1.
\end{equation}
Then we prove that
\begin{equation*}
\dim_{\rm H}(\mathcal{U}_{\tau}[\alpha])=\liminf_{k\to+\infty}\frac{\frac{1}{\tau}\log q_{k}+(\frac{1}{\tau}-1)\sum_{j=2}^{k-1}\log q_{j}}{\log q_{k}+\log q_{k+1}},\ \forall\ \tau\in\left(\frac{1}{w'},1\right)\cup(1,w').
\end{equation*}
For any $\tau\in(\frac{1}{w'},1)$, by Lemma $\ref{92}$, we know that
\begin{equation*}
\dim_{\rm H}(\mathcal{U}_{\tau}[\alpha])= \liminf\limits_{k\to+\infty}\frac{\log\left(n_{k}^{1+\frac{1}{\tau}}\prod_{j=1}^{k-1}n_{j}^{\frac{1}{\tau}}\langle n_{j}\alpha\rangle\right)}{\log(n_{k}\langle n_{k}\alpha\rangle^{-1})},
\end{equation*}
where $(n_{k})_{k=1}^{\infty}$ is the maximal subsequence of $(q_{k})_{k=1}^{\infty}$ such that $n_{k}\langle n_{k}\alpha\rangle^{\tau}<1$. Since
\begin{equation*}
\frac{1}{\tau}<w'=\liminf\limits_{k\to+\infty}\frac{\log q_{k+1}}{\log q_{k}},
\end{equation*}
we have
\begin{equation*}
q_{k}<q_{k+1}^{\tau}, \quad\ \forall\ k\gg1.
\end{equation*}
This together with $\eqref{101}$ gives
\begin{equation*}
q_{k}\langle q_{k}\alpha\rangle^{\tau}<q_{k}q_{k+1}^{-\tau}<q_{k+1}^{\tau}q_{k+1}^{-\tau}=1.
\end{equation*}
for all $k$ large enough. Thus,
\begin{equation*}
\begin{aligned}
&\quad\dim_{\rm H}(\mathcal{U}_{\tau}[\alpha])=\liminf\limits_{k\to+\infty}\frac{\log(q_{k}^{1+\frac{1}{\tau}}\prod_{j=1}^{k-1}q_{j}^{\frac{1}{\tau}}\langle q_{j}\alpha\rangle)}{\log(q_{k}\langle q_{k}\alpha\rangle^{-1})}\\
&=\liminf\limits_{k\to+\infty}\frac{\log\left(q_{k}^{1+\frac{1}{\tau}}\prod_{j=1}^{k-1}q_{j}^{\frac{1}{\tau}}q_{j+1}^{-1}\right)}{\log(q_{k}q_{k+1})}
=\liminf_{k\to+\infty}\frac{\frac{1}{\tau}\log q_{k}+(\frac{1}{\tau}-1)\sum_{j=2}^{k-1}\log q_{j}}{\log q_{k}+\log q_{k+1}}.
\end{aligned}
\end{equation*}
The second equality is due to $\eqref{101}$ and the super-exponentially increasing of $(q_{k})_{k=1}^{\infty}$. Similarly, for every $\tau\in(1,w')$, we have
\begin{equation*}
\dim_{\rm H}(\mathcal{U}_{\tau}[\alpha])=\liminf_{k\to+\infty}\frac{\frac{1}{\tau}\log q_{k}+(\frac{1}{\tau}-1)\sum_{j=2}^{k-1}\log q_{j}}{\log q_{k}+\log q_{k+1}}.
\end{equation*}
It follows from ($\ref{111}$) that for each $\tau\in(\frac{1}{w'},1)$,
\begin{equation*}
\begin{aligned}
\dim_{\rm H}(\mathcal{U}_{\tau}[\alpha])&=\liminf_{k\to+\infty}\frac{\frac{1}{\tau}\log q_{k}+(\frac{1}{\tau}-1)\sum_{j=2}^{k-1}\log q_{j}}{\log q_{k}+\log q_{k+1}}\\&\leq\liminf_{k\to+\infty}\frac{\frac{1}{\tau}\log q_{k}+(\frac{1}{\tau}-1)\log q_{k}\sum_{j=1}^{k-2}(w'-\epsilon)^{-j}}{\log q_{k}+\log q_{k+1}}\\&=\frac{\frac{1}{\tau}+(\frac{1}{\tau}-1)\frac{1}{w'-\epsilon-1}}{1+w(\alpha)}.
\end{aligned}
\end{equation*}
Furthermore, by ($\ref{111}$), for every $\tau\in(1,w')$,
\begin{equation*}
\begin{aligned}
\dim_{\rm H}(\mathcal{U}_{\tau}[\alpha])&=\liminf_{k\to+\infty}\frac{\frac{1}{\tau}\log q_{k}+(\frac{1}{\tau}-1)\sum_{j=2}^{k-1}\log q_{j}}{\log q_{k}+\log q_{k+1}}\\&\geq\liminf_{k\to+\infty}\frac{\frac{1}{\tau}\log q_{k}+(\frac{1}{\tau}-1)\log q_{k}\sum_{j=1}^{k-2}(w'-\epsilon)^{-j}}{\log q_{k}+\log q_{k+1}}\\&=\frac{\frac{1}{\tau}+(\frac{1}{\tau}-1)\frac{1}{w'-\epsilon-1}}{1+w(\alpha)}.
\end{aligned}
\end{equation*}
Since $\mathcal{U}_{1}[\alpha]\subset\mathcal{U}_{\tau}[\alpha]$ for any $\tau\in(\frac{1}{w'},1)$, we have
\begin{equation*}
\dim_{\rm H}(\mathcal{U}_{1}[\alpha])\leq\dim_{\rm H}(\mathcal{U}_{\tau}[\alpha])\leq\frac{\frac{1}{\tau}+(\frac{1}{\tau}-1)\frac{1}{w'-\epsilon-1}}{1+w(\alpha)},\ \forall\ \tau\in\left(\frac{1}{w'},1\right).
\end{equation*}
Letting $\tau\to1^{-}$, we obtain that
\begin{equation}\label{112}
\dim_{\rm H}(\mathcal{U}_{1}[\alpha])\leq\frac{1}{w(\alpha)+1}.
\end{equation}
Because $\mathcal{U}_{\tau}[\alpha]\subset\bigcup\limits_{\tau>1}\mathcal{U}_{\tau}[\alpha]$ for each $\tau\in(1,w')$, we have
\begin{equation*}
\dim_{\rm H}\left(\bigcup_{\tau>1}\mathcal{U}_{\tau}[\alpha]\right)\geq\dim_{\rm H}(\mathcal{U}_{\tau}[\alpha])\geq\frac{\frac{1}{\tau}+(\frac{1}{\tau}-1)\frac{1}{w'-\epsilon-1}}{1+w(\alpha)},\ \forall\ \tau\in(1,w').
\end{equation*}
Letting $\tau\to 1^{+}$, we obtain that
\begin{equation}\label{113}
\dim_{\rm H}\left(\bigcup_{\tau>1}\mathcal{U}_{\tau}[\alpha]\right)\geq\frac{1}{w(\alpha)+1}.
\end{equation}
The combination of ($\ref{112}$) and ($\ref{113}$) gives $$\dim_{\rm H}\left(\bigcup_{\tau>1}\mathcal{U}_{\tau}[\alpha]\right)=\dim_{\rm H}(\mathcal{U}_{1}[\alpha])=\frac{1}{w(\alpha)+1}.$$
\textrm{}
\end{proof}
Now, we are in a position to prove Theorem $\ref{120}$.
\begin{proof}[\rm \textbf{Proof of Theorem $\ref{120}$}]
Denote
\begin{equation*}
\Gamma_{\alpha}(1,1)=\{\gamma\in[0,1): S_{1}(\alpha,\gamma)<+\infty\}.
\end{equation*}
Recall that $S_{1}(\alpha,\gamma)=\sum\limits_{Q=1}^{+\infty}\min\limits_{1\leq q\leq Q}\langle q\alpha-\gamma\rangle$. By Corollary $\ref{41}$, we have
\begin{equation*}
\Omega_{\alpha}(1,1)=\Gamma_{\alpha}(1,1)\setminus\bigcup_{q\in\mathbb{N}}\bigcup_{p\in\mathbb{Z}}\{q\alpha+p\}.
\end{equation*}
Furthermore, it follows from Lemma $\ref{54}$ that
\begin{equation}\label{105}
\bigcup_{\tau>1}\mathcal{U}_{\tau}[\alpha]\subset\Gamma_{\alpha}(1,1)\subset\mathcal{U}_{1}[\alpha].
\end{equation}
By the countable stability of Hausdorff dimension and the fact that every countable set has Hausdorff dimension zero, we have
\begin{equation}\label{106}
\dim_{\rm H}(\Gamma_{\alpha}(1,1))=\dim_{\rm H}(\Omega_{\alpha}(1,1)).
\end{equation}
Combining ($\ref{105}$) and ($\ref{106}$), we obtain that
\begin{equation}\label{107}
\dim_{\rm H}\left(\bigcup_{\tau>1}\mathcal{U}_{\tau}[\alpha]\right)\leq\dim_{\rm H}(\Omega_{\alpha}(1,1))\leq\dim_{\rm H}(\mathcal{U}_{1}[\alpha]).
\end{equation}
(\rmnum{1}) The combination of ($\ref{107}$) and Lemma $\ref{93}$ gives
\begin{equation*}
\dim_{\rm H}(\Omega_{\alpha}(1,1))\leq\frac{2}{w(\alpha)+1}.
\end{equation*}
(\rmnum{2})
The upshot of ($\ref{107}$) and Lemma $\ref{304}$ is that $$\dim_{\rm H}(\Omega_{\alpha}(1,1))=\dim_{\rm H}\left(\bigcup_{\tau>1}\mathcal{U}_{\tau}[\alpha]\right)=\dim_{\rm H}(\mathcal{U}_{1}[\alpha])=\frac{1}{w(\alpha)+1}.$$
\textrm{}
\end{proof}

\section{Proof of Theorem $\ref{12}$}\label{500}
\begin{proof}[\rm \textbf{Proof of Theorem $\ref{12}$}]
Firstly, we show that
\begin{equation*}
([0,1)^{m\times n}\times[0,1)^{m})\setminus\boldsymbol{\rm Bad}(m,n)\subset\Lambda(m,n).
\end{equation*}
For each $(A,\boldsymbol{\gamma})\notin\boldsymbol{\rm Bad}(m,n)$, we have
\begin{equation*}
\liminf\limits_{\boldsymbol{q}\in\mathbb{Z}^{n},\|\boldsymbol{q}\|\to+\infty}\|\boldsymbol{q}\|^{n}\langle A\boldsymbol{q}-\boldsymbol{\gamma} \rangle^{m}=0.
\end{equation*}
So there exists a sequence $\{\boldsymbol{q}_{i}\}_{i=1}^{\infty}\subset \mathbb{Z}^{n}\setminus\{\boldsymbol{0}\}$ with $\|\boldsymbol{q}_{i}\|<\|\boldsymbol{q}_{i+1}\|$, satisfies
\begin{equation}\label{21}
\|\boldsymbol{q}_{i}\|^{n}\langle A\boldsymbol{q}_{i}-\boldsymbol{\gamma}\rangle^{m}<\frac{1}{2^{i}}
\end{equation}
for any $i\in\mathbb{N}$. Denote $\boldsymbol{q}_{0}=\boldsymbol{0}$. Define
\begin{equation}\label{22}
\psi(q)=\left(\frac{1}{2^{i+1}\|\boldsymbol{q}_{i+1}\|^{n}}\right)^{\frac{1}{m}}
\end{equation}
if $\|\boldsymbol{q}_{i}\|<q\leq\|\boldsymbol{q}_{i+1}\|$ for some $i\geq0$.
We know that the real positive function $\psi$ satisfies the following:
\begin{enumerate}[(1)]
\item $\psi$ is decreasing;
\item \begin{equation*}
  \begin{aligned}
 \sum\limits_{q=1}^{+\infty}q^{n-1}\psi(q)^{m}&=\sum\limits_{i=0}^{+\infty}\sum\limits_{q=\|\boldsymbol{q}_{i}\|+1}^{\|\boldsymbol{q}_{i+1}\|}q^{n-1}\psi(q)^{m}\\
 &\leq\sum\limits_{i=0}^{+\infty}\sum\limits_{q=\|\boldsymbol{q}_{i}\|+1}^{\|\boldsymbol{q}_{i+1}\|}\frac{1}{2^{i+1}\|\boldsymbol{q}_{i+1}\|} \\
 &\leq\sum\limits_{i=0}^{+\infty}\frac{1}{2^{i+1}}=1;
  \end{aligned}
 \end{equation*}
\item $\langle A\boldsymbol{q_{i}}-\boldsymbol{\gamma}\rangle<\psi(\|\boldsymbol{q_{i}}\|)$ for any $i\geq1$ by \eqref{21} and \eqref{22}.
\end{enumerate}
The above (1) and (2) gives that $\psi\in\mathcal{C}$, which implies that $(A,\boldsymbol{\gamma})\in W_{m,n}(\psi)$ together with (3). Thus $(A,\boldsymbol{\gamma})\in\Lambda(m,n)$. \\ Secondly, we prove that
\begin{equation*}
\Lambda(m,n)\subset([0,1)^{m\times n}\times[0,1)^{m})\setminus\boldsymbol{\rm Bad}(m,n).
\end{equation*}
For any $(A,\boldsymbol{\gamma})\in\Lambda(m,n)$, we have that there exists a decreasing function $\psi$ with $\sum\limits_{q=1}^{+\infty}q^{n-1}\psi(q)^{m}<+\infty$, such that $\langle A\boldsymbol{q}-\boldsymbol{\gamma}\rangle<\psi(\|\boldsymbol{q}\|)$ for infinitely many $\boldsymbol{q}\in\mathbb{Z}^{n}$. By Lemma $\ref{54}$, for any $\epsilon>0$, we have $\psi(q)^{m}<\epsilon\cdot q^{-n}$ whenever $q$ is large enough. Therefore we have
\begin{equation*}
\langle A\boldsymbol{q}-\boldsymbol{\gamma}\rangle^{m}<\epsilon \|\boldsymbol{q}\|^{-n}
\end{equation*}
for infinitely many $\boldsymbol{q}\in\mathbb{Z}^{n}$. It means that
\begin{equation*}
\liminf_{\boldsymbol{q}\in\mathbb{Z}^{n},\|\boldsymbol{q}\|\to+\infty}\|\boldsymbol{q}\|^{n}\langle A\boldsymbol{q}-\boldsymbol{\gamma}\rangle^{m}\leq\epsilon.
\end{equation*}
By the arbitrariness of $\epsilon$, we obtain that
\begin{equation*}
\liminf_{\boldsymbol{q}\in\mathbb{Z}^{n},\|\boldsymbol{q}\|\to+\infty}\|\boldsymbol{q}\|^{n}\langle A\boldsymbol{q}-\boldsymbol{\gamma}\rangle^{m}=0.
\end{equation*}
Hence, $(A,\boldsymbol{\gamma})\notin\boldsymbol{\rm Bad}(m,n)$.
\textrm{}
\end{proof}

\section{Proof of Theorem $\ref{14}$ and Corollary $\ref{15}$}\label{600}
In order to prove Theorem $\ref{14}$ and Corollary $\ref{15}$, we need the following lemma, which is called Baire category theorem.
\begin{lemma}[Baire category theorem]\label{16}
Assume that $(X,d)$ is a complete metric space. If $\{U_{n}\}_{n=1}^{\infty}$ is a sequence of dense and open sets in $X$, then $\bigcap\limits_{n=1}^{+\infty}U_{n}$ is dense in $X$. In particular, $X$ is of second category.
\end{lemma}
\begin{remark}
The proof of Baire category theorem can be found in \cite[Theorem 48.2]{JRM}.
\end{remark}
\begin{proof}[\rm \textbf{Proof of Theorem $\ref{14}$}]
(\rmnum{1}) Firstly, we prove that $\mathcal{C}(A,\boldsymbol{\gamma})$ is a $G_{\delta}$ set. In order to show that $\mathcal{C}(A,\boldsymbol{\gamma})$ is a $G_{\delta}$ set, we can write the set $\mathcal{C}(A,\boldsymbol{\gamma})$ as
\begin{equation*}
\mathcal{C}(A,\boldsymbol{\gamma})=\bigcap\limits_{k=1}^{+\infty}\mathcal{C}_{k}(A,\boldsymbol{\gamma}),
\end{equation*}
where
\begin{equation*}
\mathcal{C}_{k}(A,\boldsymbol{\gamma}):=\{\psi\in\mathcal{C}:{\rm the\ number\ of}\ \boldsymbol{q}\in\mathbb{Z}^{n}\ {\rm such\ that}\ \langle A \boldsymbol{q}-\boldsymbol{\gamma}\rangle<\psi(\|\boldsymbol{q}\|)\ {\rm is\ at\ least}\ k\}.
\end{equation*}
The left is to show that $\mathcal{C}_{k}(A,\boldsymbol{\gamma})$ is an open set in $\mathcal{C}$ for any $k\geq1$. In fact, for any $\psi\in\mathcal{C}_{k}(A,\boldsymbol{\gamma})$, there exists $\boldsymbol{q}_{1},\ldots,\boldsymbol{q}_{k}$ with $\langle A\boldsymbol{q}_{i}-\boldsymbol{\gamma}\rangle<\psi(\|\boldsymbol{q}_{i}\|)$, for all $1\leq i\leq k$. Let
\begin{equation*}
\epsilon=\frac{1}{2}\min\{\psi(\|\boldsymbol{q}_{i}\|)-\langle A\boldsymbol{q}_{i}-\boldsymbol{\gamma}\rangle: 1\leq i\leq k\}.
\end{equation*}
Since function $x\mapsto x^{\frac{1}{m}}, x\in[0,+\infty)$ is uniformly continuous, there exists $\delta>0$, such that for any $\varphi\in B(\psi, \delta)$, where $B(\psi,\delta)=\{\varphi\in\mathcal{C}: d(\varphi,\psi)<\delta\}$, we have
\begin{equation*}
|\varphi(\|\boldsymbol{q}_{i}\|)-\psi(\|\boldsymbol{q}_{i}\|)|<\epsilon
\end{equation*}
for each $1\leq i\leq k$. Therefore
\begin{equation*}
\varphi(\|\boldsymbol{q}_{i}\|)>\psi(\|\boldsymbol{q}_{i}\|)-\epsilon>\langle A\boldsymbol{q}_{i}-\boldsymbol{\gamma}\rangle
\end{equation*}
for every $1\leq i\leq k$, which means that $\varphi\in\mathcal{C}_{k}(A,\boldsymbol{\gamma})$. By the arbitrariness of $\varphi$, we obtain that $B(\psi,\delta)\subset\mathcal{C}_{k}(A,\boldsymbol{\gamma})$. Thus, $\mathcal{C}_{k}(A,\boldsymbol{\gamma})$ is an open set in $\mathcal{C}$. \\ Secondly, we show that $\mathcal{C}(A,\boldsymbol{\gamma})$ is dense in $\mathcal{C}$. That is, for any $\psi\in\mathcal{C}$ and any $\epsilon>0$, we need to prove that $B(\psi,\epsilon)\cap\mathcal{C}(A,\boldsymbol{\gamma})\neq\emptyset$. Because
\begin{equation*}
\liminf\limits_{\boldsymbol{q}\in\mathbb{Z}^{n},\|\boldsymbol{q}\|\to+\infty}\|\boldsymbol{q}\|^{n}\langle A\boldsymbol{q}-\boldsymbol{\gamma}\rangle^{m}=0,
\end{equation*}
there exists a sequence $\{\boldsymbol{q}_{i}\}_{i=1}^{\infty}\subset \mathbb{Z}^{n}\setminus\{\boldsymbol{0}\}$, such that
\begin{equation}\label{23}
\|\boldsymbol{q}_{i}\|^{n}\langle A\boldsymbol{q}_{i}-\boldsymbol{\gamma}\rangle^{m}<\frac{\epsilon}{2^{i}}
\end{equation}
for every $i\geq1$. Let $\boldsymbol{q}_{0}=\boldsymbol{0}$, define
\begin{equation}\label{24}
\varphi(q)=\left(\psi(q)^{m}+\frac{\epsilon}{2^{i+1}\|\boldsymbol{q}_{i+1}\|^{n}}\right)^{\frac{1}{m}},
\end{equation}
if $\|\boldsymbol{q}_{i}\|<q\leq\|\boldsymbol{q}_{i+1}\|$ for some $i\geq0$. We know that the real positive function $\varphi$ satisfies the following:
\begin{enumerate}[(1)]
\item $\varphi$ is decreasing;
\item \begin{equation*}
    \begin{aligned}
    \sum\limits_{q=1}^{+\infty}q^{n-1}|\varphi(q)^{m}-\psi(q)^{m}|&=\sum\limits_{i=0}^{+\infty}\sum\limits_{q=\|\boldsymbol{q}_{i}\|+1}^{\|\boldsymbol{q}_{i+1}\|}q^{n-1}|\varphi(q)^{m}-\psi(q)^{m}|\\
    &\leq\sum\limits_{i=0}^{+\infty}\sum\limits_{q=\|\boldsymbol{q}_{i}\|+1}^{\|\boldsymbol{q}_{i+1}\|}\frac{\epsilon}{2^{i+1}\|\boldsymbol{q}_{i+1}\|}\\ &< \sum\limits_{i=0}^{+\infty}\frac{\epsilon}{2^{i+1}}=\epsilon;
    \end{aligned}
  \end{equation*}
\item
$\langle A\boldsymbol{q_{i}}-\boldsymbol{\gamma}\rangle<\varphi(\|\boldsymbol{q_{i}}\|)$ for each $i\geq1$ by \eqref{23} and \eqref{24}.
\end{enumerate}
The above (1) and (2) give that $\varphi\in B(\psi,\epsilon)$, which implies that $\varphi\in B(\psi,\epsilon)\cap\mathcal{C}(A,\boldsymbol{\gamma})$. Thus, $\mathcal{C}(A,\boldsymbol{\gamma})$ is dense in $\mathcal{C}$ by the arbitrariness of $\psi$ and $\epsilon$.


(\rmnum{2}) In order to prove that $\mathcal{C}(A,\boldsymbol{\gamma})$ is not a $F_{\sigma}$ set, we show that ${\rm Int}(\mathcal{C}(A,\boldsymbol{\gamma}))=\emptyset$,
where ${\rm Int}(\mathcal{C}(A,\boldsymbol{\gamma}))$ denote the set of all interior points of $\mathcal{C}(A,\boldsymbol{\gamma})$. For any $\psi\in\mathcal{C}(A,\boldsymbol{\gamma})$ and any $\epsilon>0$, we would construct a function $\varphi\in B(\psi,\epsilon)\setminus\mathcal{C}(A,\boldsymbol{\gamma})$. Since $\sum\limits_{q=1}^{+\infty}q^{n-1}\psi(q)^{m}<+\infty$, there exists $N\geq2$, such that $\sum\limits_{q=N}^{+\infty}q^{n-1}\psi(q)^{m}<\epsilon$. Define
\begin{equation}\label{25}
\begin{aligned}
\varphi(q)=\begin{cases} \psi(q), &\text{if}\ q<N, \\  0, &\text{if}\ q\geq N. \end{cases}
\end{aligned}
\end{equation}
We know that the real non-negative function $\varphi$ satisfies the following:
\begin{enumerate}[(1)]
\item $\varphi$ is decreasing;
\item \begin{equation*}
 \begin{aligned}
 \sum\limits_{q=1}^{+\infty}q^{n-1}\left|\psi(q)^{m}-\varphi(q)^{m}\right|&=\sum\limits_{q=N}^{+\infty}q^{n-1}\left[\psi(q)^{m}-\varphi(q)^{m}\right] \\
 &=\sum\limits_{q=N}^{+\infty}q^{n-1}\psi(q)^{m}<\epsilon.
 \end{aligned}
\end{equation*}
\item $\varphi(\|\boldsymbol{q}\|)\leq\langle A\boldsymbol{q}-\boldsymbol{\gamma}\rangle$ for all $\boldsymbol{q}\in\mathbb{Z}^{n}$ with $\|\boldsymbol{q}\|\geq N$ by $\eqref{25}$.
\end{enumerate}
By (1) and (2), we immediately obtain that $\varphi\in B(\psi,\epsilon)$, which implies that $\varphi\in B(\psi,\epsilon)\setminus\mathcal{C}(A,\boldsymbol{\gamma})$ together with (3). It follows that $\psi\notin{\rm Int}(\mathcal{C}(A,\boldsymbol{\gamma}))$. By the arbitrariness of $\psi$, we obtain that ${\rm Int}(\mathcal{C}(A,\boldsymbol{\gamma}))=\emptyset$. Finally, we prove the conclusion by contradiction. Suppose that $\mathcal{C}(A,\boldsymbol{\gamma})$ is a $F_{\sigma}$ set in $\mathcal{C}$, then
\begin{equation*}
\mathcal{C}(A,\boldsymbol{\gamma})=\bigcup_{i=1}^{+\infty}F_{i},
\end{equation*}
where $\{F_{i}\}_{i=1}^{\infty}$ is a sequence of closed sets in $\mathcal{C}$. Since ${\rm Int}(\mathcal{C}(A,\boldsymbol{\gamma}))=\emptyset$, we obtain that ${\rm Int}(F_{i})=\emptyset$ for each $i\in\mathbb{N}$, which implies that $\mathcal{C}(A,\boldsymbol{\gamma})$ is of first category. However, by Theorem $\ref{14}$ (\rmnum{3}), we know that $\mathcal{C}(A,\boldsymbol{\gamma})$ is of second category, we get a contradiction. Therefore, $\mathcal{C}(A,\boldsymbol{\gamma})$ is not a $F_{\sigma}$ set in $\mathcal{C}$.
\textrm{}
\end{proof}

\begin{proof}[\rm \textbf{Proof of Corollary $\ref{15}$}]
For every $i\geq1$, we write $\mathcal{C}(A_{i},\boldsymbol{\gamma}_{i})$ as
\begin{equation*}
\mathcal{C}(A_{i},\boldsymbol{\gamma}_{i})=\bigcap_{k=1}^{+\infty}\mathcal{C}_{k}(A_{i},\boldsymbol{\gamma}_{i}),
\end{equation*}
where
\begin{equation*}
\mathcal{C}_{k}(A_{i},\boldsymbol{\gamma}_{i}):=\{\psi\in\mathcal{C}:{\rm the\ number\ of}\ \boldsymbol{q}\in\mathbb{Z}^{n}\ {\rm such\ that}\ \langle A_{i}\boldsymbol{q}-\boldsymbol{\gamma}_{i}\rangle<\psi(\|\boldsymbol{q}\|)\ {\rm is\ at\ least}\ k\}.
\end{equation*}
By Theorem $\ref{14}$, we know that $\mathcal{C}_{k}(A_{i},\boldsymbol{\gamma}_{i})$ is dense and open in $\mathcal{C}$ for all $k\geq1$. Thus,
\begin{equation*}
\bigcap_{i=1}^{+\infty}\mathcal{C}(A_{i},\boldsymbol{\gamma}_{i})=\bigcap_{i=1}^{+\infty}\bigcap_{k=1}^{+\infty}\mathcal{C}_{k}(A_{i},\boldsymbol{\gamma}_{i})
\end{equation*}
is a intersection of countable many dense and open sets in $\mathcal{C}$. By Lemma $\ref{16}$,
\begin{equation*}
\bigcap_{i=1}^{+\infty}\mathcal{C}(A_{i},\boldsymbol{\gamma}_{i})
\end{equation*}
is dense and of second category.
\textrm{}
\end{proof}

\subsection*{Acknowledgments}
The authors thank Professor Sanju Velani for constructive discussions, which are very helpful in improving the quality of the article. This work was supported by National Key R\&D Program of China (No. 2024YFA1013700), NSFC 12271176 and Guangdong Natural Science Foundation 2024A1515010946.

\end{document}